\documentclass[11pt]{article}
\usepackage{amsfonts,amsmath,amssymb,verbatim}
\usepackage{hyperref,amsthm}
\usepackage{epsfig}
\usepackage{graphicx}

\newtheorem{theorem}{Theorem}[section]
\newtheorem{proposition}[theorem]{Proposition}
\newtheorem{lemma}[theorem]{Lemma}
\newtheorem{claim}[theorem]{Claim}
\newtheorem{corollary}[theorem]{Corollary}

\newtheorem{notation}[theorem]{Notation}

\newtheorem{conjecture}[theorem]{Conjecture}
\newtheorem{problem}[theorem]{Problem}

\newtheorem*{theorem*}{Theorem}
\theoremstyle{definition}
\newtheorem{remark}{Remark}[section]
\newtheorem*{definition}{Definition}

\newcommand{\beq}[1]{\begin{equation}\label{#1}}
\newcommand{\enq}[0]{\end{equation}}

\newcommand{\mn}[0]{\medskip\noindent}

\newcommand{\ud}{\textup{d}}

\newcommand{\A}[0]{{\cal A}}
\newcommand{\B}[0]{{\cal B}}
\newcommand{\F}[0]{{\cal F}}

\newcommand{\D}[0]{{\cal D}}
\newcommand{\G}[0]{{\cal G}}
\newcommand{\h}[0]{{\cal H}}
\newcommand{\I}[0]{{\cal I}}
\newcommand{\J}[0]{{\cal J}}
\newcommand{\p}[0]{{\cal P}}
\newcommand{\s}[0]{{\cal S}}

\newcommand{\e}{{\mathbb{AK}_{t,\zeta}^{n}}}
\newcommand{\emm}{{\mathbb{AK}_{t,\zeta}^{m}}}
\global\long\def\f{\mathcal{F}}

\global\long\def\pn{\mathcal{P}\left(\left[n\right]\right)}
\global\long\def\g{\mathcal{G}}
\global\long\def\s{\mathcal{S}}

\global\long\def\j{\mathcal{J}}

\global\long\def\p{\mathcal{P}}

\global\long\def\h{\mathcal{H}}

\global\long\def\a{\mathcal{A}}
\global\long\def\b{\mathcal{B}}

\global\long\def\e{\mathbb{E}}

\DeclareMathOperator{\Bin}{Bin}

\usepackage[usenames]{color}

\begin{document}

\title{Stability for the Complete Intersection Theorem, and the Forbidden Intersection Problem of Erd\H{o}s and S\'os}

\author{
David Ellis\thanks{School of Mathematical Sciences, Queen Mary, University of London, Mile End Road, London, E1 4NS, United Kingdom. E-mail: {\tt d.ellis@qmul.ac.uk}.}, \mbox{ }
Nathan Keller\thanks{Department of Mathematics, Bar Ilan University, Ramat Gan, Israel.
E-mail: {\tt nathan.keller27@gmail.com}. Research supported by the Israel Science Foundation (grant no.
402/13), the Binational US-Israel Science Foundation (grant no. 2014290), and by the Alon Fellowship.}
\mbox{ } and Noam Lifshitz\thanks{Department of Mathematics, Bar Ilan University, Ramat Gan, Israel.
E-mail: {\tt noamlifshitz@gmail.com}.}
}

\maketitle

\begin{abstract}
A family $\F$ of sets is said to be {\em $t$-intersecting} if $|A \cap B| \geq t$ for any $A,B \in \F$.
The seminal Complete Intersection Theorem of Ahlswede and Khachatrian (1997)
 gives the maximal size $f(n,k,t)$
of a $t$-intersecting family of $k$-element subsets of
$[n]=\{1,2,\ldots,n\}$, together with a characterisation
of the extremal families.

The {\em forbidden intersection problem}, posed by Erd\H{o}s and S\'{o}s in 1971, asks for a determination of the maximal
size $g(n,k,t)$ of a family $\F$ of $k$-element subsets of $[n]$ such that $|A \cap B| \neq t-1$
for any $A,B \in \F$.

In this paper, we show that for any fixed $t \in \mathbb{N}$, if
$o(n) \leq k \leq n/2-o(n)$, then $g(n,k,t)=f(n,k,t)$. In combination with prior results, this solves the problem of Erd\H{o}s and S\'{o}s for any
constant $t$, except for in the ranges $n/2-o(n) < k < n/2+t/2$ and $k < 2t$.

One key ingredient of the proof is the following sharp `stability' result for the Complete Intersection
Theorem: if $k/n$ is bounded away from $0$ and $1/2$, and $\F$ is a $t$-intersecting family of $k$-element
subsets of $[n]$ such that $|\F| \geq f(n,k,t) - O(\binom{n-d}{k})$, then there exists
a family $\G$ such that
$\G$ is extremal for the Complete Intersection Theorem, and $|\F \setminus \G| = O(\binom{n-d}{k-d})$.
This proves a conjecture of Friedgut (2008). We prove the result by combining classical `shifting' arguments
with a `bootstrapping' method based upon an isoperimetric inequality.

Another key ingredient is a Ôweak regularity lemmaÕ for families of $k$-element subsets
of $[n]$, where $k/n$ is bounded away from 0 and 1. This states that any such family $\F$
is approximately contained within a ÔjuntaÕ, such that the restriction of $\F$ to each subcube
determined by the junta is ÔpseudorandomÕ in a certain sense.
\begin{comment}
We believe that this Ôweak
regularity lemmaÕ will find further applications.
\end{comment}
\end{abstract}

\section{Introduction}
We write $[n]:=\{1,2,\ldots,n\}$, and $\binom{[n]}{k}:=\{A \subset [n]: |A|=k\}$. If $X$ is a set, we write $\p(X)$ for the power-set of $X$. A family $\F \subset \p([n])$
(i.e., a family of subsets of $[n]$) is said to be \emph{increasing} if
$A\supset B\in\F$ implies $A\in \F$, and \emph{intersecting} if for
any $A,B \in \F$, we have $A \cap B \neq \emptyset$. For $t \in \mathbb{N}$, $\F$ is said to be \emph{t-intersecting} if
for any $A,B \in \F$, we have $|A \cap B| \geq t$.
A \emph{dictatorship} is a family
of the form $\{S| i \in S\} := \D_i$ for some $i \in [n]$, and
a \emph{$t$-umvirate} is a family of the form
$\{S| B \subset S\} =: \s_B$, for some $B \in \binom{[n]}{t}$. If $X$ is a set, we write $\textrm{Sym}(X)$ for the symmetric group on $X$. We say that two families $\mathcal{F},\mathcal{G} \subset \p([n])$ are {\em isomorphic} if there exists a permutation $\sigma \in \textrm{Sym}([n])$ such that $\mathcal{G} = \{\sigma(S):\ S \in \F\}$; in this case, we write $\F \cong \G$.

%For $r \geq 3$, $F$
%is called $r$-wise $t$-intersecting if for any $A_1,A_2,\ldots,A_r \in F$, $|A_1 \cap
%\ldots \cap A_r| \geq t$. Finally, families $F_1,F_2,\ldots,F_r$ are called cross
%$r$-wise $t$-intersecting if for any $A_1 \in F_1,A_2 \in F_2,\ldots,A_r \in F_r$,
%$|A_1 \cap \ldots \cap A_r| \geq t$.
%\end{definition}

\medskip

The classical Erd\H{o}s-Ko-Rado theorem~\cite{EKR} determines the maximal size of an intersecting
family $\F \subset \binom{[n]}{k}$.
\begin{theorem}[Erd\H{o}s-Ko-Rado, 1961]\label{Thm:EKR}
Let $k < n/2$, and let $\F \subset \binom{[n]}{k}$ be an intersecting family. Then $|\F| \leq {{n-1}\choose{k-1}}$.
Equality holds if and only if $\F$ is a dictatorship.
\end{theorem}

This theorem is the starting-point of an entire subfield of extremal combinatorics, concerned with bounding the sizes of families of sets, under various intersection requirements on sets in the family. Such results are often called \emph{Erd\H{o}s-Ko-Rado type results}. For more background and history on Erd\H{o}s-Ko-Rado type results, we refer the reader to the surveys \cite{DF83,FT16,MV15} and the references therein.

Also in \cite{EKR}, Erd\H{o}s, Ko and Rado showed that for $n$ sufficiently large depending on $k$ and $t$, the maximal size of a $t$-intersecting family $\F \subset \binom{[n]}{k}$ is ${{n-t}\choose{k-t}}$. For general $(n,k,t) \in \mathbb{N}^3$, we write $f(n,k,t)$ for this maximum. The determination of $f(n,k,t)$ for a general triple $(n,k,t) \in \mathbb{N}^3$ remained a major open problem for more than three decades. Frankl \cite{fr-conj} conjectured that for any $(n,k,t) \in \mathbb{N}^3$, there exists $r \in \mathbb{N} \cup \{0\}$ such that the family
$$\F_{n,k,t,r} := \{S \in \binom{[n]}{k}: |S \cap [t+2r]| \geq t+r\}$$
is a $t$-intersecting subfamily of $\binom{[n]}{k}$ of maximal size. Following partial results by Frankl \cite{fr-conj} and Wilson \cite{Wilson}, Frankl's conjecture was eventually proved by Ahlswede and Khachatrian~\cite{AK}:
\begin{theorem}[Ahlswede-Khachatrian `Complete Intersection Theorem', 1997]\label{Thm:AK}
Let $n,k,t \in \mathbb{N}$, and let $\F \subset \binom{[n]}{k}$ be a $t$-intersecting family. Then $|\F| \leq \max_r |\F_{n,k,t,r}|$, and equality holds only if $\F$ is isomorphic to $\F_{n,k,t,r}$ for some $r \geq 0$. In particular, if $n \geq (t+1)(k-t+1)$, then
$|\F| \leq |\F_{n,k,t,0}|={{n-t}\choose{k-t}}$.
\end{theorem}

In 1971, Erd\H{o}s and S\'os (see \cite{erdos1975problems}) raised the question of what happens if the $t$-intersecting condition is replaced by
the weaker condition that no two sets
in $\f$ have intersection of size exactly $t-1$.
\begin{problem}[Erd\H{o}s-S\'os, 1971]
\label{prob:es}
For $n,k,t \in \mathbb{N}$, what is the maximal size $g\left(n,k,t\right)$ of a family $\F \subset \binom{\left[n\right]}{k}$ such
that $|A \cap B| \ne t-1$ for any $A,B \in \F$?
\end{problem}
The first significant progress on this problem was made by Frankl and
F\"uredi \cite{frankl1985forbidding} in 1985; they showed that $g\left(n,k,t\right)=\binom{n-t}{k-t}$ provided $k\ge2t$ and $n$ is sufficiently large depending on $k$ and $t$.
\begin{theorem}[Frankl-F\"uredi, 1985]
\label{thm:Frankl-Furedi} For any $k,t \in \mathbb{N}$ such that $k\ge2t$, there exists $n_0(k,t) \in \mathbb{N}$ such that the following holds. Let $n \geq n_0(k,t)$, and let $\F \subset \binom{\left[n\right]}{k}$ such that $\left|A\cap B\right| \neq t-1$ for any $A,B \in \F$. Then $\left|\f\right|\le\binom{n-t}{k-t}$. Equality holds if and only if $\F$ is a $t$-umvirate.
\end{theorem}
As pointed out in \cite{frankl1985forbidding}, the hypothesis $k\ge2t$ in Theorem \ref{thm:Frankl-Furedi} is necessary, in the sense that $g\left(n,k,t\right)>\binom{n-t}{k-t}$ if $k<2t$, for all sufficiently large $n$. 

In 2007, Keevash, Mubayi and Wilson~\cite{kmw} presented a complete solution of the case $(k=4,t=2)$, for all $n \in \mathbb{N}$. Recently, the second and third authors~\cite{chvatal} proved that for any $t \in \mathbb{N}$, there exists $c = c(t) >0$ such that $g(n,k,t) = {n-t \choose k-t}$ whenever $1/c \leq k \leq cn$; the extremal families in this range are precisely the $t$-umvirates. However, no general result was known for $k=\Theta(n)$, and in particular, in any case where the the extremal families are not $t$-umvirates.

In this paper, we prove the following Ahlswede-Khachatrian type result for the Erd\H{o}s-S\'os problem, resolving the latter in the case where $t$ is fixed, $k/n$ is bounded away from $0$ and $1/2$ and $n$ is large.
\begin{theorem}
\label{thm:main-result-es}
For any $t \in \mathbb{N}$ and any $\zeta >0$, there exists $n_0(t,\zeta) \in \mathbb{N}$ such that the following holds. Let $n\ge n_{0}\left(t,\zeta,\right)$, let $n,k \in \mathbb{N}$ with $\zeta n < k <(\tfrac{1}{2}-\zeta)n$, and let $\f \subset \binom{\left[n\right]}{k}$
such that no two sets in $\f$ have intersection of size $t-1$. Then $\left|\f\right|\le f\left(n,k,t\right)$, and equality holds only if $\f$ is isomorphic to $\f_{n,k,t,r}$, for some $r\ge0$.
\end{theorem}
%We remark that the second and third authors recently proved in \cite{chvatal} that for any $t \in \mathbb{N}$, there exists $c = c(t) >0$ such that $g(n,k,t) = {n-t \choose k-t}$ whenever $1/c \leq k \leq cn$; the extremal families in this range are precisely the $t$-umvirates. 
Combined with the aforementioned previous results, this resolves the Erd\H{o}s-S\'{o}s problem for all triples $(n,k,t) \in \mathbb{N}^3$ such that $2t \leq k \leq (1/2 - \zeta)n$ and $n$ is sufficiently large depending on $t$ and $\zeta$, for any $\zeta >0$, giving also a characterisation of the extremal families in these cases. Since Problem~\ref{prob:es} is trivial for $k \geq (n+t)/2$ (as in this case, any two distinct sets in $\binom{[n]}{k}$ have intersection of size at least $t$), the only remaining cases are $k<2t$ and $n-o(n)<k<(n+t)/2$. 

Our first main tools in proving Theorem \ref{thm:main-result-es} is the following sharp `stability' result for the Ahlswede-Khachatrian theorem, which 
%we believe to be of independent interest.
in itself proves a conjecture of Friedgut~\cite{Friedgut08} from 2008.
\begin{theorem}\label{thm:main-ak-stability}
For any $t \in \mathbb{N}$ and any $\zeta >0$, there exists $C = C(t,\zeta) >0$ such that the following holds. Let $n,k,d\in\mathbb{N}$ such that $\zeta n < k < (\tfrac{1}{2}-\zeta)n$, and let $\F \subset \binom{[n]}{k}$ be a $t$-intersecting
family such that $|\F|> f(n,k,t)-\frac{1}{C}\binom{n-d}{k}$. Then there exists $\G \subset \binom{[n]}{k}$ isomorphic to some $\F_{n,k,t,r}$,
such that $\left|\F\backslash\G\right|<C\binom{n-d}{k-d}$, where $r \leq C$.
\end{theorem}

%As we will outline in the next subsection, this theorem implies a strengthened version of a conjecture of Friedgut. 
Theorem~\ref{thm:main-ak-stability} is tight, up to a factor depending only on $t$ and $\zeta$, as evidenced by the families
\begin{align*}
\mathcal{H}_{n,k,t,r,d} & :=\left\{ A \in \binom{[n]}{k}\,:\,|A \cap [t+2r]| \geq t+r,\, A \cap \left\{t+2r+1,\ldots,t+2r+d\right\} \ne\emptyset\right\} \\
 & \cup\left\{ A \in \binom{[n]}{k}\,:\,|A \cap [t+2r]|=t+r-1,\,\left\{t+2r+1,\ldots,t+2r+d\right\}\subset A\right\},
\end{align*}
for sufficiently large $n$ and $d$.

Our second main tool in proving Theorem \ref{thm:main-result-es} is a `weak regularity lemma' for families $\F\subset \binom{[n]}{k}$, where $k/n$ is bounded away from 0 and 1. This states that such a family $\F$ is approximately contained within a `junta' (i.e.\ a family depending upon few coordinates), such that the restriction of $\F$ to each subcube determined by the junta is `pseudorandom' in a certain sense.
\begin{comment} We believe that this `weak regularity lemma' is also of interest in its own right, and will find further applications. \end{comment}
To state it formally, we need some more definitions.

For $0 \leq k \leq n$, we write $\mu$ for the uniform measure on $\binom{[n]}{k}$, i.e.\
$$\mu(\F) := |\F|/\binom{n}{k},\quad \F \subset \binom{[n]}{k}.$$
For $\F \subset \binom{[n]}{k}$ and $B \subset J \subset [n]$, we write $\F_J^B = \{S \setminus B:\ S \in \F,\ S \cap J =B\} \subset \binom{[n] \setminus J}{k - |B|}$. We call these families `slices' of $\F$.

For $J \subset [n]$, we say that a family $\J \subset \binom{[n]}{k}$ is a {\em $J$-junta} if there exists a family $\G \subset \p(J)$ such that $S \in \J$ if and only if $S \cap J \in \G$, for all $S \in \binom{[n]}{k}$. In this case, we say that $\F$ is the {\em $J$-junta generated by $\G$}, and we write $\J = \langle \G \rangle$.

The crucial definition is as follows. For $\delta >0$ and $h \in \mathbb{N}$, we say a family $\F \subset \binom{[n]}{k}$ is {\em $(\delta,h)$-slice-quasirandom} if for any $J \subset [n]$ with $|J| \leq h$, and any $B \subset J$, we have $|\mu(\F_J^B) - \mu(\F)| < \delta$. In other words, for every $B \subset J$, $\mu(\F_J^B)$ is close to $\mu(\F)$, which of course would be the expected value of $\mu(\F_J^B)$ if $\F$ were a random subset of $\binom{[n]}{k}$ with density $\mu(\F)$. (We emphasise that we regard $\F_J^B$ as a subset of $\binom{[n] \setminus J}{k - |B|}$, and so $\mu(\F_J^B) = |\F_J^B|/\binom{n-|J|}{k-|B|}$.)

Here, then, is our `weak regularity lemma'.

\begin{theorem}
\label{thm:weak-reg}
For any $\zeta,\delta,\epsilon\in\left(0,1\right)$ and $h \in \mathbb{N}$,
there exists $j=j(\zeta,\delta,h,\epsilon)\in\mathbb{N}$ and $n_0 = n_0(\zeta,\delta,h,\epsilon) \in \mathbb{N}$ such that the following holds. Let $n \geq n_0$, let $\zeta n < k < \left(1-\zeta\right)n$,
and let $\f \subset \binom{\left[n\right]}{k}$. Then
there exist a set $J\subset \left[n\right]$ with $|J| \leq j$,
and a subset $\G \subset \p(J)$, such that:
\begin{enumerate}
\item[(1)] $\mu\left(\f\backslash\left\langle \G \right\rangle \right)<\epsilon.$

\item[(2)] For each $B \in \G$, the family $\f_{J}^{B}$ is a $\left(\delta,h\right)$-slice-quasirandom family satisfying $\mu(\f_{J}^{B})>\frac{\epsilon}{2}$.
\end{enumerate}
\end{theorem}

Informally, Theorem \ref{thm:weak-reg} says that for any family $\f\subset \binom{\left[n\right]}{k}$, there exists a set $J\subset [n]$ such that the $2^{|J|}$ slices $\{\F_J^B:\ B \subset J\}$ can be divided into two `types': `good' slices for which $\F_J^B$ is `random-like' and not too small, and `bad' slices, with small total size. Alternatively, condition (1) says that $\F$ is almost contained within the $J$-junta $\langle \G \rangle$, and condition (2) says that for each subcube of the form $\{S \subset [n]:\ S \cap J = B\}$ with $B \in \G$, the restriction of $\F$ to that subcube is `random-like' and not too small.

The remainder of this paper is structured as follows. In subsection \ref{subsec:prior}, we discuss some related prior work on stability for Erd\H{o}s-Ko-Rado type theorems. In subsection \ref{subsec:reg}, we discuss regularity lemmas in general, and compare some previously-known ones with ours. In subsection \ref{subsec:sketch}, we sketch the methods we will use to prove our main theorems. In Section~\ref{sec:preliminaries} we present some of the known results and techniques which we will use in our proofs -- concerning juntas, influences, shifting, cross-intersecting
families, and reduction to the `biased measure' setting. In Section \ref{sec:ak-stability}, we prove Theorem \ref{thm:main-ak-stability}, our stability result for the Ahlswede-Khachatrian theorem. In Section \ref{sec:reg}, we prove Theorem \ref{thm:weak-reg}, our weak regularity lemma for hypergraphs of linear uniformity. In Section \ref{sec:main-result-es}, we prove Theorem \ref{thm:main-result-es}, our main result on the Erd\H{o}s-S\'os problem. We conclude with some open problems in Section \ref{sec:conc}.

\subsection{Stability for the Erd\H{o}s-Ko-Rado and Ahlswede-Khachatrian theorems}
\label{subsec:prior}
Over the last fifty years, several authors have obtained \emph{stability results} for the Erd\H{o}s-Ko-Rado (EKR) and
the Ahlswede-Khachatrian (AK) theorems. In general, a stability result asserts that if the size of a family is `close' to the maximum possible size, then that family is `close' (in an appropriate sense) to an extremal family.

One of the first such results is due to Hilton and Milner~\cite{HM67}, who showed in 1967 that if the size of an intersecting family is \emph{very} close to ${{n-1}\choose{k-1}}$, then the family is {\it contained} in a dictatorship. A similar result for the complete intersection theorem in the domain $n \geq (k-t+1)(t+1)$ was obtained in 1996 by Ahlswede and Khachatrian~\cite{AK96}. A simpler proof of the latter result was presented by Balogh and Mubayi~\cite{BM08}, and an alternative result of the same class was obtained by Anstee and Keevash~\cite{AK06}.

For families whose size is not very close to the maximum, Frankl~\cite{Frankl87} obtained in 1987 a strong stability version of the EKR theorem which implies that if an intersecting family $\mathcal{F}$ satisfies $|\mathcal{F}| \geq (1-\epsilon){{n-1}\choose{k-1}}$, then there exists a dictatorship $\mathcal{D}$ such that $|\F \setminus \mathcal{D}| = O(\epsilon^{\log_{1-p}p}) {{n}\choose{k}}$, where $p \approx k/n$. Frankl's result is tight and holds not only for  $|\mathcal{F}|$ close to ${{n-1}\choose{k-1}}$ but rather whenever $|\mathcal{F}| \geq 3{n-2\choose k-2}-2{n-3 \choose k-3}$. Proofs of somewhat weaker results using entirely different techniques were later presented by Dinur and Friedgut~\cite{DF09}, Friedgut~\cite{Friedgut08} and Keevash~\cite{Keevash08}. In~\cite{KM10}, Keevash and Mubayi used Frankl's result to prove an EKR-type theorem on set systems that do not contain a simplex or a cluster. Recently, a different notion of
stability for the EKR theorem was suggested by Bollob\'as, Narayanan and Raigorodskii \cite{BNR16}; this has been studied in several subsequent papers (e.g.,~\cite{BBN16+,DK16+}).

The case of the AK theorem appeared much harder. The first stability result was obtained by Friedgut~\cite{Friedgut08}, who showed in 2008 that for any $\epsilon \geq \sqrt{(\log n) / n}$, $\zeta >0$ and $\zeta n < k<(1/(t+1)- \zeta)n$, if a $t$-intersecting $\F \subset \binom{[n]}{k}$ satisfies $|\F| \geq f(n,k,t)(1-\epsilon)$, then there exists a $t$-umvirate $\G$ such that $|\F \setminus \G| = O_{t,\zeta}(\epsilon) {{n}\choose{k}}$. The proof of Friedgut uses Fourier analysis and spectral methods. In~\cite{EKL16+}, the authors proved a strong version of Friedgut's result, which asserts that under the conditions of Friedgut's theorem, $|\F \setminus \G| = O(\epsilon^{\log_{1-p}p}) {{n}\choose{k}}$ (where $p \approx k/n$) for some $t$-umvirate $\G$, and showed that it is tight by an explicit example. The main technique of~\cite{EKL16+} is to utilize isoperimetric inequalities on the hypercube.

\medskip All the results described above apply only in the so-called `principal domain' $k<n/(t+1)$, in which the extremal example has
the simple structure of a $t$-umvirate. In the general case, where the extremal examples are the more complex families $\F_{n,k,t,r}$,
no stability result has been obtained so far (to the best of our knowledge). This situation resembles the history of the `exact' results, where Theorem~\ref{Thm:AK}
was proved for $k<n/(t+1)$ by Wilson~\cite{Wilson} in 1984, but it was 13 more years until the general case was
resolved by Ahlswede and Khachatrian.

\medskip
The main conjecture stated in Friedgut's 2008 paper~\cite{Friedgut08} is that his stability result holds for all
$\zeta n< k< (1/2 -\zeta)n$. To state the conjecture, we need some additional explanation.

A direct computation shows that for any $\beta \in (0,1/2)$ and any $t \in \mathbb{N}$, there is either a unique value of $r$ or two consecutive values of $r$ that
asymptotically maximize $|\F_{n, \lfloor \beta n \rfloor,t,r}|$ (as $n \to \infty$). We say that $\beta$ is {\em non-singular for $t$} if there is a unique such value of $r$, which we then denote by $r^*=r^*(\beta,t)$. Otherwise we say that $\beta$ is {\em singular for $t$}, and we let $r^*$ and $r^*+1$ be the two
extremal values of $r$.
\begin{conjecture}\cite[Conjecture~4.1]{Friedgut08}\label{Conj:Friedgut}
Let $t \in \mathbb{N}$, let $\zeta >0$, let $\beta \in [\zeta,1/2-\zeta]$ be non-singular for $t$, let $\epsilon >0$, and let $k= \lfloor \beta n \rfloor$.
If $\F \subset \binom{[n]}{k}$ is a $t$-intersecting family such that $|\F| \geq (1 - \epsilon)|\F(n, k, t, r^*)|$,
then there exists a set $B \subset [n]$ of size $t + 2r^*$ such that $|\{A \in \F : |A \cap B| \geq t + r^*\}| \geq (1 - O_{t,\zeta}(\epsilon))|\F|$. If $\beta$ is singular for $t$, then either the above holds or the corresponding statement for $r^*+ 1$ holds.
\end{conjecture}

It is easy to see that our Theorem \ref{thm:main-ak-stability} implies Conjecture~\ref{Conj:Friedgut}. In fact, Theorem \ref{thm:main-ak-stability} implies that the conclusion of Conjecture \ref{Conj:Friedgut} can be strengthened to
$|\{A \in \F : |A \cap B| \geq t + r^*\}| \geq (1 - O_{t,\zeta}(\epsilon^{\log_{1-\beta}\beta}))|\F|$ (or $|\{A \in \F : |A \cap B| \geq t + r^*+1\}| \geq (1 - O_{t,\zeta}(\epsilon^{\log_{1-\beta}\beta}))|\F|)$, if $\beta$ is singular for $t$). Combining Theorem \ref{thm:main-ak-stability} with Theorem 1.5 in \cite{EKL16+}, one obtains a stability version of the Complete Intersection Theorem for all $k<(1/2-\zeta)n$, sharp up to a constant factor depending only upon $t$ and $\zeta$.

\subsection{Regularity lemmas}
\label{subsec:reg}
Over the last forty years, `regularity lemmas' have been crucial ingredients in a wide variety of important results in Combinatorics. The earliest such lemma is the classical Szemer\'edi Regularity Lemma for graphs \cite{sz}. Roughly speaking, this states that the vertex-set of any large, dense graph $G$ can be partitioned into a bounded number of parts $V_0,V_1,\ldots,V_N$, where $V_0$ is small, $|V_1| = |V_2| = \ldots = |V_N|$, and for most pairs $\{i,j\} \in \binom{[N]}{2}$, the induced bipartite subgraph $G[V_i,V_j]$ of $G$ with parts $V_i$ and $V_j$ is `pseudorandom', in the sense that for any large subsets $A \subset V_i$ and $B \subset V_j$, $e(G[A,B])$ is close to what one would expect if $G[V_i,V_j]$ were a random bipartite graph with the same edge-density. 

Since Szemer\'edi proved his celebrated lemma, `regularity lemmas' for a wide variety of combinatorial structures have been obtained. Broadly speaking, such lemmas state that a sufficiently large combinatorial structure can be partitioned into a bounded number of pieces which are `random-like' in the sense that they behave roughly as if they were `random' structures of the same density, together with a small amount of `waste' or `noise'. 

Gowers \cite{gow-reg}, and independently R\"odl and Skokan \cite{rs}, proved analogues of Szemer\'edi's regularity lemma for hypergraphs of fixed uniformity.
Green \cite{green} proved a regularity lemma for Boolean functions on Abelian groups; the $\mathbb{Z}_2^n$ case of this states that for any $f:\mathbb{Z}_2^n \to \{0,1\}$, there exists a subgroup $H \leq \mathbb{Z}_2^n$ of bounded index, such that on most cosets $C$ of $H$, the restriction of $f$ to $C$ is `pseuodorandom' in the sense of having small non-trivial Fourier coefficients. Jones \cite[Theorem 2]{jones} (following unpublished work of O'Donnell, Servedio, Tan and Wan) and independently Mossel \cite[Lemma 5.3]{mossel}, proved variants of the $\mathbb{Z}_2^n$-case of Green's regularity lemma, where the subgroup $H$ is of the form $\{x \in \mathbb{Z}_2^n:\ x_i = 0\ \forall i \in S\}$ for some $S \subset [n]$, and the notions of pseudorandomness are somewhat weaker than Green's. Mossel's notion of pseudorandomness is termed `resilience': a function $f:\mathbb{Z}_2^n \to \mathbb{R}$ is said to be {\em $(r,\alpha)$-resilient} if $|\mathbb{E}[f|\{S = z\}]-\mathbb{E}[f]| \leq \alpha$ for all $S \subset [n]$ with $|S| \leq r$ and all $z \in \mathbb{Z}_2^S$. Our notion of pseudorandomness is precisely the analogue of resilience, for Boolean functions on $\binom{[n]}{k}$.\footnote{The notion of pseudorandomness used by Jones~\cite{jones} is stated in terms of the noise operator, and is somewhat stronger than resilience.} 
%We remark that the results of \cite{jones,mossel} apply also to $[0,1]$-valued functions, not just to Boolean functions.

Our `weak regularity lemma' (Theorem \ref{thm:weak-reg}) is a natural variant of the above results of Jones~\cite{jones} and of Mossel~\cite{mossel}, for Boolean functions on $\binom{[n]}{k}$, where $k/n$ is bounded away from 0 and 1. In fact, the regularity lemma of Jones et al can be generalised straightforwardly to the $p$-biased measure case, and it is not too hard to deduce our weak regularity lemma from this generalisation. Indeed, given a family $\F \subset \binom{[n]}{k}$, one takes $p = k/n + \sqrt{(\log n)/n}$ and applies the aforementioned generalisation to the function
$$f:\{0,1\}^n \to [0,1];\ x \mapsto \begin{cases} \Pr_{T \in \binom{S(x)}{k}} [T \in \F] & \text{ if } |S(x)| \geq k;\\ 0 & \text{ if } |S(x)| < k,\end{cases}$$
where $S(x) = \{i \in [n]:\ x_i=1\}$, and the above probability refers to $T$ being chosen uniformly at random from $\binom{S(x)}{k}$. We give a different, more self-contained proof of Theorem \ref{thm:weak-reg}, one which we believe to be more natural. (We remark that it does not seem possible to deduce Theorem \ref{thm:weak-reg} from the lemma of Mossel, even though his regularity lemma applies to the $p$-biased measure; this is because we require a junta $\langle \G \rangle$ such that {\em all} the slices of $\F$ corresponding to $\G$ are highly pseudorandom.)

We call Theorem \ref{thm:weak-reg} a `weak regularity lemma' because, as with the so-called `weak regularity lemma' of Frieze and Kannan \cite{fk}, our notion of $(\eta,h)$-slice-quasirandomness does not imply a general `counting lemma' for hypergraphs with a fixed number of edges, in the sense of \cite{gow-reg,nrs} (see Remark \ref{remark:counting}). It should be noted, however, that our proof of Theorem \ref{thm:weak-reg} gives $j,n_0 = 2 \uparrow \uparrow 1/(\zeta^{O(h)} \delta^2 \epsilon)$, where for $m >0$, $2 \uparrow \uparrow m$ denotes a tower of twos of height $\lceil m \rceil$. Most known regularity lemmas come with tower-type bounds or worse, and such bounds have been shown to be necessary in many cases by Gowers \cite{gow-nec}, and by Conlon and Fox \cite{cf}. One exception is the aforementioned `weak regularity lemma' of Frieze and Kannan, where the bound on the number of parts is only exponential.

\begin{comment}
is not, to the best of our knowledge, implied by any prior result. In particular, the results of \cite{jones} and \cite{mossel} do not apply in our setting, as any `layer' $\binom{[n]}{k}$ occupies only a $o(1)$-fraction of the whole discrete cube, $\pn \cong \mathbb{Z}_2^n$. Moreover, our proof of Theorem \ref{thm:weak-reg} is somewhat different to the proofs in \cite{jones} and \cite{mossel}, as it uses an `entropy increment' strategy (as in \cite{fox}), rather than the `energy increment' strategy which is more usual in proofs of regularity lemmas. We remark that Friedgut and Regev \cite{fr} recently used a very similar `entropy increment' strategy to prove that for a large family of product-graphs, and for Kneser graphs, a set spanning a small proportion of edges of the graph can be made into an independent set by removing a small proportion of the vertices of the graph. However, our approach is slightly different from theirs: they work with the $p$-biased measure on $\pn$ and use results in this setting to deduce results about families of $k$-element sets, whereas we work directly with families of $k$-element sets.
\end{comment}

Theorem \ref{thm:weak-reg} can be seen as a `regularity lemma' for hypergraphs of linear uniformity (i.e., uniformity linear in the number of vertices). We remark that the hypergraph regularity lemmas of Gowers \cite{gow-reg} and of R\"odl and Skokan \cite{rs} do not apply to hypergraphs of linear uniformity. The notions of `pseudorandomness' in these lemmas are very different from ours; unlike ours, both notions admit general `counting lemmas' \cite{gow-reg,nrs} giving asymptotic estimates on the number of copies of a hypergraph with a fixed number of edges in a `pseudorandom' hypergraph.

\subsection{A sketch of our methods}
\label{subsec:sketch}
\subsubsection*{Stability for $t$-intersecting families}
We first outline our proof of Theorem \ref{thm:main-ak-stability}. As in several previous works on stability for Erd\H{o}s-Ko-Rado type theorems (e.g.,~\cite{DF09,EKL16+,Friedgut08}), it is more convenient for us to work first with
the \emph{biased measure} on $\p([n])$, rather than with the uniform measure on $\binom{[n]}{k}$. Hence, we first consider $t$-intersecting
families $\F \subset \p([n])$, and seek to maximize their $p$-biased measure $\mu_p(\F)$, defined by
$\mu_p(\F) := \sum_{S \in \F} p^{|S|} (1-p)^{n-|S|}$. The biased version of the AK theorem (presented clearly in \cite{Filmus13}) is as follows.
\begin{theorem}[Biased AK Theorem]
\label{thm:ak-biased}
Let $t \in \mathbb{N}$, let $0 < p < 1/2$, and let $\F \subset \p([n])$ be a $t$-intersecting family. Then $\mu_p(\F) \leq f(n,p,t):= \max_r \mu_p(\F_{n,t,r})$,
where $\F_{n,t,r} := \{S \subset [n]:
|S \cap [t+2r]| \geq t+r\}$, and equality holds iff $\F$ is isomorphic to one of the $\F_{n,t,r}$ families. In particular, if $p < 1/(t+1)$, then $\mu_p(\F) \leq \mu_p(\tilde{\F}_{n,t,0}) = p^t$, with equality iff $\F \cong \F_{n,t,0}$.
\end{theorem}
\noindent We prove the following stability version of Theorem \ref{thm:ak-biased}.
\begin{theorem}\label{thm:biased-ak-stability}
For any $t\in\mathbb{N}$ and any $\zeta>0$, there exists $C=C(t,\zeta)>0$ such that the following holds. Let $p\in\left[\zeta,\frac{1}{2}-\zeta\right]$, and let $\epsilon >0$. If $\F \subset \p([n])$ is a $t$-intersecting family such that $\mu_{p}\left(\F\right)\ge f(n,p,t)\left(1-\epsilon\right)$, then there exists a family $\G$ isomorphic to some $\F_{n,t,r}$, such that $\mu_{p}\left(\F\backslash \G \right)\le
C\epsilon^{\log_{1-p}p}$.

\mn In particular, if $\frac{r^*}{t+2r^*-1}+\zeta<p<\frac{r^*+1}{t+2r^*+1}-\zeta$ for some $r^* \in \mathbb{N}$, then
the above holds with $\G \cong \F_{n,t,r^*}$.
\end{theorem}

\noindent Theorem \ref{thm:biased-ak-stability} is tight up to a factor depending only upon $t$ and $\zeta$, as evidenced by the families
\begin{align*}
\tilde{\mathcal{H}}_{n,t,r,s} & =\left\{ A\subset \mathcal{P}\left(\left[n\right]\right)\,:\,|A \cap [t+2r]| \geq t+r,\, A \cap \left\{t+2r+1,\ldots,t+2r+s\right\} \ne\emptyset\right\} \\
 & \cup\left\{ A\subset \mathcal{P}\left(\left[n\right]\right)\,:\,|A \cap [t+2r]|=t+r-1,\,\left\{t+2r+1,\ldots,t+2r+s\right\}\subset A\right\},
\end{align*}
for sufficiently large $n$ and $s$. The computation showing this is presented in section~\ref{sec:sub:tightness}.
%Indeed, we have
%\[
%\mu_{p}\left(\mathcal{G}_{t,r,s}\right)= \mu_p(\F_{n,t,r}\left(1-\left(1-p\right)^{s}\right)+ {{t+2r}\choose{t+r-1}} p^{t+r-1}\left(1-p\right)^{r+1}p^{s}
%\]
%and
%\[
%\mu_{p}\left(\mathcal{F}_{t,s}\backslash\mathcal{S}_{\left[t\right]}\right)={{t+2r}\choose{t+r-1}} p^{t+r-1}\left(1-p\right)^{r+1}p^{s},
%\]
%which corresponds to the assertion of Theorem~\ref{thm:Main Stability} with $\epsilon=tp^{s}$ and $2^{t}-1$ replaced by $t$. Actually, we
%conjecture that the $\F_{s,t}$'s are \emph{the} extremal examples (that is, $2^t-1$ in the theorem should be replaced by $t$) but were not
%able to prove that so far.
%A direct computation shows that for $p\in\left(\frac{r^*}{t+2r^*-1},\frac{r^*+1}{t+2r^*+1}\right)$, the unique
%maximizer of $\mu_p(\F_{n,t,r})$ is $r=r^*$. As a result, Theorem~\ref{thm:Main Stability} easily implies the
%following version:
%\begin{theorem}
%Let $t, r^*\in\mathbb{N}$, $\zeta>0$, $p\in\left[\frac{r^*}{t+2r^*-1}+\zeta,\frac{r^*+1}{t+2r^*+1}-\zeta\right]$.
%There exist $C=C(t,\zeta), \epsilon_0=\epsilon_0(t,\zeta)$ such that the following holds for all
%$\epsilon<\epsilon_0$. Let $\F \subset \p([n])$ be a $t$-intersecting family. If $\mu_{p}\left(\F\right)\ge
%f(n,p,t)\left(1-\epsilon\right)$, then there exists $\G$ isomorphic to $\F_{n,t,r^*}$ such that
%$\mu_{p}\left(\F\backslash \G \right)\le C\epsilon^{\log_{1-p}p}$.
%\end{theorem}

\medskip \noindent Theorem~\ref{thm:biased-ak-stability} follows from combination of three ingredients:
\begin{itemize}
\item A \textbf{bootstrapping lemma} showing that if a $t$-intersecting $\F$ is {\it somewhat close} to some
$\F_{n,t,r}$ then it must be {\it very close} to that $\F_{n,t,r}$. More precisely, there exists $c>0$ such that
if $\mu_p(\F \setminus \F_{n,t,r})>c$, and if $\G$ is another $t$-intersecting family which is a {\it small modification}
of $\F$, in the sense that $\mu_p(\F \setminus \G) < c/2$, then $\mu_p(\G \setminus \F_{n,t,r})>c$. Hence,
there is a `barrier' which one cannot cross while making only small modifications.

\item A \textbf{shifting argument} showing that given a $t$-intersecting family $\F$, one can transform it into
a \emph{junta} (i.e., a function that depends on only $O(1)$ coordinates) $\tilde{\F}$ with
$\mu_p(\tilde{\F}) \geq \mu_p(\F)$ by a series of small modifications.

\item An observation that if a $t$-intersecting junta $\tilde{\F}$ satisfies $\mu_p(\tilde{\F}) > f(n,p,t)(1-\epsilon)$ for
a sufficiently small $\epsilon$, then it must be isomorphic to one of the $\F_{n,t,r}$'s.
\end{itemize}
While the shifting part is based on `classical' shifting arguments summarized in~\cite{Filmus13}, the bootstrapping
relies on a recently introduced isoperimetric argument~\cite{EKL16+}. It seems that the combination of
the classical shifting tools with isoperimetry is the main novelty of the proof.

\medskip

Next, we deduce Theorem \ref{thm:main-ak-stability} from Theorem~\ref{thm:biased-ak-stability}. For this, we first use a standard reduction from the $k$-uniform setting to the biased-measure setting to obtain a `weak' stability theorem for $t$-intersecting families of $k$-element sets. Then, we `bootstrap' this weak stability result to obtain Theorem \ref{thm:main-ak-stability}, using an argument recently introduced in \cite{EKL16+}, relying on the Kruskal-Katona theorem and some extremal results on cross-intersecting families.

\subsection*{Our weak regularity lemma}
In common with many other regularity-type results, the proof of our `weak regularity lemma' (Theorem \ref{thm:weak-reg}) uses a potential argument. We define a non-positive potential function $\phi:\ \p(\tbinom{[n]}{k}) \times \p([n]) \to [-1/e,0]$ such that:
\begin{itemize}
\item $\phi(\F,S) \leq \phi(\F,S')$ if $S \subset S'$;
\item If $k/n$ is bounded away from $0$ and $1$, then for any $\F \subset \binom{[n]}{k}$ and any $S\subset \left[n\right]$, {\em either} there exists an $S$-junta $\G \subset \p(S)$ such that $\G$ satisfies the conclusion of Theorem \ref{thm:weak-reg} (with $J = S$), {\em or} there exists a set $S' \supset S$ that is not much larger than $S$, such that $\phi\left(\f,S'\right)$
is significantly larger than $\phi\left(\f,S\right)$.
\end{itemize}

Using our potential function $\phi$, we prove the existence of a junta satisfying the conclusion of Theorem \ref{thm:weak-reg}, as follows. We start by setting $S = \emptyset$. By the second property of $\phi$, either there exists an $S$-junta $\G_S \subset \p(S)$ such that $\G$ satisfies the conclusion of Theorem \ref{thm:weak-reg} (with $J = S$), or else there exists a set $S' \supset S$ that is not much larger than $S$, such that $\phi\left(\f,S'\right)$
is significantly larger than $\phi\left(\f,S\right)$. In the former case, we are done; in the latter case, we replace $S$ by $S'$ and repeat. Since $\phi$ is bounded from above by $0$, the former case must occur after a bounded number of steps.

Roughly speaking, our proof yields a sequence of juntas (depending on successively larger, nested sets of coordinates) approximately containing $\F$. At each stage, we have a set of coordinates $S$ and an $S$-junta $\J_S$ approximately containing $\F$. If a significant part of the mass of $\F$ lies in subcubes (determined by $S$) on which $\F$ is not sufficiently random-like, then $\phi$ may be increased by a significant amount. When $\phi$ stops increasing by a significant amount, we have our required junta.

The function $\phi$ is entropy-related, so loosely speaking, our method can be viewed as an `entropy-increment' strategy, as opposed to the more common `energy-increment' strategy. The former turns out to be slightly cleaner in our case, but the exact choice of $\phi$ is not particularly important, as we do not seek to optimise $j=j(\zeta,\delta,h,\epsilon)$ in the conclusion of Theorem \ref{thm:weak-reg}.

\subsection*{Families with a forbidden intersection: a stability result and an exact result}
Given a family $\F \subset \binom{[n]}{k}$ that contains no two sets with intersection of size $t-1$ (where $k/n$ is bounded away from $0$ and $1$), we first apply Theorem \ref{thm:weak-reg} to find a junta $\J = \langle \G \rangle$ such that $\F$ is approximately contained within $\J$, and for each $B \in \G$, the family $\F_J^B$ is highly slice-quasirandom and not too small. Next, we prove a lemma about pairs of slice-quasirandom families (Lemma \ref{lem:qrc}): if $k/n$ and $l/n$ are bounded away from $0$ and $1/2$, then for any pair $\A \subset \binom{[n]}{k}$ and $\B \subset \binom{[n]}{l}$ of highly slice-quasirandom families which are not too small, and for any fixed $t \in \mathbb{N}$, there exist $A \in \A$ and $B \in \B$ such that $|A \cap B| = t-1$. In other words, slice-quasirandomness allows one to achieve any fixed intersection-size. (We prove this by using the slice-quasirandomness property to reduce to the case $t=1$, i.e.\ to a statement about cross-intersecting families; this can be tackled using known techniques, again translating from the uniform setting to the biased-measure setting.)

We use Lemma \ref{lem:qrc} to show that the junta $\J$ must in fact be $t$-intersecting, so $\F$ is approximately contained within a $t$-intersecting family (namely, $\J$). It follows that if $\F$ has size close to $f(n,k,t)$, then the $t$-intersecting family $\J$ also has size close to $f(n,k,t)$. We then apply Theorem \ref{thm:main-ak-stability} to deduce that $\J$ (and therefore $\F$) is close in symmetric difference to one of the families $\F_{n,k,t,r}$, where $r$ is bounded.

To summarize, the above argument yields a stability result for families with a forbidden intersection: if $k/n$ is bounded away from $0$ and $1/2$, and $\F \subset \binom{[n]}{k}$ is a family such that no two sets in $\F$ have intersection of size $t-1$ and $|\F|$ is close to $f(n,k,t)$, then $\F$ has small symmetric difference with some $\F_{n,k,t,r}$, where $r$ is bounded. Finally, we can show that if $\F$ is a small perturbation of one of the families $\F_{n,k,t,r}$ (where $r$ is bounded), then $\F$ is smaller than $\F_{n,k,t,r}$, hence deducing Theorem \ref{thm:main-result-es} from our stability result.

\section{Prior results and techniques}
\label{sec:preliminaries}

Our proofs use several previous results and techniques from extremal and probabilistic combinatorics.
In this section we present the notation, definitions, and previous results and techniques that will be used in the
sequel. As some of the results were not proved in the form we use them, we present their proofs here for sake of
completeness. The reader may find it helpful to look through this section briefly at first, and then go back to specific
results when they are used in the sequel.

\subsection{Notation}
\label{sec:sub:notations}

If $\F \subset \p([n])$, then we write $\bar{\F} := \{[n] \setminus A:\ A \in \F\}$, and we write $\F^* := \p([n]) \setminus \bar{\F}$ for the {\em dual} of $\F$. For each $k \in [n] \cup \{0\}$, the {\em $k$th layer} of $\F$ is $\F \cap \binom{[n]}{k}=:\F^{(k)}$. If $\F \subset \p([n])$, we define the {\em increasing family generated by} $\F$ to be
\[\F^{\uparrow} = \{A \subset [n]:\ \exists B \in \F \textrm{ with } B \subset A\},\]
i.e.\ it is the minimal increasing family that contains $\F$.

\mn In sections \ref{sec:preliminaries}-\ref{sec:proof}, we will often regard $p \in (0,1)$ as fixed, and we will sometimes suppress it from our notation.

\mn For fixed $(n,p,t)$ with $n,t \in \mathbb{N}$ and $0 < p < 1/2$, $\e$ will denote the collection of extremal families for the biased AK theorem corresponding to $(n,p,t)$ (i.e., the collection of all $t$-intersecting $\F$ with $\mu_p(\F)=f(n,p,t)$). Note that for any $\A \in \e$ and any $p \in [\zeta,1/2-\zeta]$,
we have $\mu_p(\A) = \Theta_{t,\zeta}(1)$.

\mn For any $\F$, we denote by $\mu_p(\F \setminus \e)$ the `minimal distance' $\min_{\G \in \e} \mu_p(\F \setminus \G)$.

\mn For $k \in [n]$, $(\e)^{(k)}$ denotes the $k$th layers of elements of $\e$, i.e., $\{\A \cap \binom{[n]}{k}: \A \in \e\}$.
For $\F \subset \binom{[n]}{k}$, we denote by $|\F \setminus (\e)^{(k)}|$ the `minimal distance' $\min_{\G \in (\e)^{(k)}} |\F \setminus \G|$.

\mn For $B \subset J \subset [n]$, we define
$\F_{J}^{B}:=\left\{ A\subset \left[n\right]\backslash J\,:\, A\cup B\in\F\right\}$.

\mn For a fixed $c>0$, we say that $\G$ is a \emph{$c$-small modification} of $\F$ if
$\mu_p(\F \setminus \G) \leq c$.

\mn A {\em uniform family} is a subset of $\binom{[n]}{k}$ for some $k \in [n] \cup \{0\}$. When $k$ is understood, we will sometimes write $\mu$ for the uniform measure on $\binom{[n]}{k}$, defined by
$$\mu(\A) = \frac{|\A|}{\binom{n}{k}}\quad \forall \A \subset \binom{[n]}{k}.$$

\mn The \emph{lower shadow} of a uniform family $\A \subset \binom{[n]}{k}$ is defined as $\partial(\A) := \{B \in \binom{[n]}{k-1}: \exists A \in \A \text{ with }B \subset A\}$.

\mn All logs in this paper are to the base $e$, unless otherwise indicated by a subscript, e.g.\ $\log_b$.

\mn We use the (now standard) `asymptotic notation', as follows. If $f = f(x)$ and $g = g(x)$ are non-negative functions, we write $f = O(g)$ if there exists $C>0$ such that $f(x) \leq Cg(x)$ for all $x$. We write $f = \Omega(g)$ if there exists $c>0$ such that $f(x) \geq cg(x)$ for all $x$, and we write $f = \Theta(g)$ if $f = O(g)$ and $f = \Omega(g)$. If $f = f(x; \alpha)$ and $g = g(x; \alpha)$ are non-negative functions, then we write $f = O_{\alpha} (g)$ if for all $\alpha$, there exists $C=C(\alpha)>0$ such that $f(x;\alpha) \leq Cg(x;\alpha)$ for all $x$. Similarly, we use the notation $f = \Omega_{\alpha}(g)$ and $f = \Theta_{\alpha}(g)$. (Here, we view $\alpha$ as a parameter; note that $\alpha$ may be vector-valued.)

\subsection{The $p$-biased measure, influences and juntas}
\label{sec:sub:influences}

Let $0 < p < 1$ and let $n \in \mathbb{N}$. The {\em $p$-biased measure} on $\p([n])$ is defined by
\[
\mu_{p}(\F)=\sum_{A\in\f}p^{\left|A\right|}\left(1-p\right)^{n-\left|A\right|} \quad \forall \F \subset \pn.
\]

\begin{definition}
Let $\F \subset \p([n])$ and $i \in [n]$. The {\em set of $i$-influential elements with respect to $\F$} is
\[
\I_i(\F) := \{S: |\{S,S \triangle \{i\} \} \cap \F| = 1\}.
\]
The {\em influence of the $i$th coordinate on $\F$ with respect to $\mu_p$} is $I^p_i(\F) := \mu_p(\I_i(\F))$. The
{\em total influence of $\F$ with respect to $\mu_p$} is $I^p(\F) := \sum_{i=1}^n I_i(\F)$. When there is no risk of confusion, we will often suppress $p$ from this notation, writing $I_i(\F) = I^p_i(\F)$ and $I(\F) = I^p(\F)$.
\end{definition}

Influences play an important role in a variety of applications in combinatorics, theoretical computer science, mathematical physics and social choice theory (see e.g.\ the survey~\cite{Kalai-Safra}).

We need the following well-known isoperimetric inequality for total influence w.r.t. $\mu_p$; this appears for example in \cite{Kahn-Kalai}.
\begin{theorem}
\label{thm:skewed-iso}
Let $\F \subset \p([n])$ be an increasing family, and let $0 < p <1$. Then
\begin{equation} \label{eq:skewed-iso} pI^p(\F) \geq \mu_p(\F) \log_p (\mu_p(\F)). \end{equation}
\end{theorem}

We also need the well-known Margulis-Russo lemma (due independently to Margulis \cite{margulis} and Russo \cite{russo}), which relates the total influence of an increasing family $\F$ to the derivative of the function $p \mapsto \mu_p(\F)$.

\begin{lemma}[Margulis-Russo Lemma]
Let $\f \subset \pn$ be an increasing family and let $0 < p_0 < 1$. Then
\[
\frac{\textup{d} \mu_{p}\left(\f\right)}{\textup{d}p} \Big|_{p=p_0} = I^{p_0}\left(\f\right).
\]
\end{lemma}

\noindent We will use the following consequence of the Margulis-Russo Lemma and Theorem \ref{thm:skewed-iso}
(this appears e.g. in \cite{Grimmett}, Theorem 2.38).
\begin{lemma}
\label{lemma:mono-decreasing}
Let $\F \subset \p([n])$ be increasing. Then the function $p \mapsto \log_p(\mu_p(\F))$ is monotone non-increasing on $(0,1)$.
\end{lemma}
\begin{proof}
Let $f(p):= \log_p (\mu_p(\F)) = \log \mu_p(\F) / \log p$. We have
\begin{equation}\label{eq:derivative}
\frac{\textup{d}f}{\textup{d}p} = \frac{\frac{1}{\mu_p(\F)} \frac{\textup{d} \mu_p(\F)}{\textup{d} p} \log p - \frac{1}{p}\log (\mu_p(\F)) }{(\log p)^2}
= \frac{\frac{1}{\mu_p(\F)} I^p(\F)\log p - \frac{1}{p}\log (\mu_p(\F)) }{(\log p)^2} \leq 0,
\end{equation}
using the Margulis-Russo Lemma and (\ref{eq:skewed-iso}).
\end{proof}

\begin{definition}
If $J \subset [n]$ and $\F \subset \p([n])$, $\F$ is said to be a {\em $J$-junta} if it depends only upon the coordinates in $J$ --- formally, if there exists $\mathcal{G} \subset \p(J)$ such that $S \in \F$ if and only if $S \cap J \in \G$, for all $S \subset [n]$. A family $\F \subset \p([n])$ is said to be a {\em $j$-junta} if it is a $J$-junta for some $J \in \binom{[n]}{j}$. Note that $\F \subset \p([n])$ is a $j$-junta if and only if $I_i(\F)=0$ for at least $n-j$ coordinates $i \in [n]$ (provided $0 < p < 1$).
\end{definition}

Friedgut's junta theorem \cite{Friedgut98} states that a family $\F \subset \p([n])$ with bounded total influence w.r.t.\ the $\mu_{p}$ measure can be approximated by a junta depending upon a bounded number of coordinates.
\begin{theorem}[Friedgut's Junta Theorem]
\label{thm:fj}
For any $\zeta>0$, there exists $C = C(\zeta)>0$ such that the following holds.
Let $\zeta \leq p \leq 1-\zeta$, let $\epsilon >0$, and let $\F\subset \pn$. Then there exists a $j$-junta
$\j \subset \pn$ such that $\mu_{p}\left(\f\Delta\j\right)<\epsilon$, where $j \leq \exp(C I^p(\F) / \epsilon)$.
\end{theorem}
Friedgut's proof of Theorem \ref{thm:fj} is based upon the Fourier-analytic proof
of the Kahn-Kalai-Linial theorem \cite{KKL}.

We will also need another, simpler relation between influences and juntas: if all influences of an increasing family
are large, then the family must be a junta.
\begin{proposition}\label{Prop:non-juntas have small minimal influence}
Let $\F \subset \p([n])$ be increasing, let $0 < p < 1$ and let $c>0$. If $\min_i(I^p_i(\F)) \geq c$, then $n\le\frac{1}{c^{2}p\left(1-p\right)}$.
\end{proposition}
%\noindent The proof of Proposition~\ref{Prop:non-juntas have small minimal influence} is easy, using standard tools
%of discrete Fourier analysis. We present it here for sake of completeness.
\begin{proof}
Let $f=\sum_{S \subset [n]} \hat f(S) \chi_S$ be the Fourier expansion of the characteristic function $f=1_{\F}$,
with respect to the measure $\mu_p$. An easy calculation presented in~\cite[Proposition~8.45]{O'Donnell14} shows
that
\[
I^p\left(\F\right)=\sum_{i=1}^{n}\frac{\hat{f}\left(\{i\}\right)}{\sqrt{p\left(1-p\right)}}.
\]
By the Cauchy-Schwarz inequality and Parseval's identity,
\[
I^p\left(\F\right)=\sum_{i=1}^{n}\frac{\hat{f}\left(\{i\}\right)}{\sqrt{p\left(1-p\right)}} \le \frac{\sqrt{n}\sqrt{\sum_{i=1}^{n}\hat{f}\left(\{i\}\right)^{2}}}{\sqrt{p\left(1-p\right)}}
\le \sqrt{\frac{n}{p\left(1-p\right)}}\left\Vert f\right\Vert _{2} \le
\sqrt{\frac{n}{p\left(1-p\right)}}.
\]
If $\min_i(I^p_i(\F)) \geq c$, we have $cn \leq I^p(\F) \leq \sqrt{\frac{n}{p\left(1-p\right)}}$. Rearranging
yields the assertion.
\end{proof}

\medskip \noindent Another observation we use is that the measure of a $t$-intersecting junta $\F$ cannot be
`too close' to $f(n,p,t)$, unless $\F$ is isomorphic to one of the $\F_{n,t,r}$ families (which are, of course, juntas).
\begin{proposition}\label{lem:juntas cannot get too close}
Let $t \in \mathbb{N}$, let $\zeta >0$, and let $m \in \mathbb{N}$. There exists $\epsilon=\epsilon\left(t,\zeta,m\right)>0$
such that if $0 < p \leq 1/2-\zeta$, and $\F \subset \p\left(\left[m\right]\right)$ is a $t$-intersecting family satisfying $\mu_{p}\left(\F\right) \geq
f(m,p,t)\left(1-\epsilon\right)$, then $\F \in \mathbb{AK}_{p,t}^{m}$.
\end{proposition}
\begin{proof}
By Theorem \ref{thm:ak-biased}, for any $r \in \mathbb{N} \cup \{0\}$ and any $\frac{r}{t+2r-1} \leq p \leq \frac{r+1}{t+2r+1}$, any $t$-intersecting $\F \in \p([m]) \setminus \emm$
satisfies $\mu_{p}\left(\mathcal{F}\right)< \mu_{p}\left(\mathcal{F}_{t,r}^{m}\right)$. Consider the set
\begin{align*}
C^m_r =\{&\mu_{q}\left(\F_{m,t,r}\right)-\mu_{q}\left(\G\right):\ \G \in \p([m]) \setminus \mathbb{AK}_{q,t}^{m},\ \G \mbox{ is \ensuremath{t}-intersecting},\\
& \frac{r}{t+2r-1} \leq q \leq \frac{r+1}{t+2r+1}\}.
\end{align*}
As $m$ is fixed, $C^m_r$ is compact and all its elements are positive. Hence, $c^m_r = \min (C^m_r) > 0$.
We have
\[
[0,1/2-\zeta] \subset \bigcup_{r=0}^\ell \left[\frac{r}{t+2r-1},\frac{r+1}{t+2r+1} \right],
\]
for some $\ell=\ell\left(t,\zeta\right) \in \mathbb{N}$. Let $\epsilon = \min_{0 \leq r \leq \ell} c^m_r$. It is clear by the choice
of $\epsilon$ that for all $p \in (0,1/2-\zeta]$, if $\F\subset \p\left(\left[m\right]\right)$ is a $t$-intersecting family
with $\mu_{p}\left(\F\right) \geq f(m,p,t)\left(1-\epsilon\right) > f(m,p,t)-\epsilon$ then we must have
$\F \in \mathbb{AK}_{p,t}^{m}$. This completes the proof.
\end{proof}

\subsection{Shifting}
\label{sec:sub:shifting}

Shifting (a.k.a. `compression') is one of the most classical techniques in extremal combinatorics (see, e.g.,~\cite{Frankl87b}).
\begin{definition}
For $\F \subset \p([n])$, a set $A \in \F$, and $1 \leq j< i \leq n$, the \emph{shifting operator} $\s_{ij}$ is defined as follows: $\mathcal{S}_{ij}\left(A\right)=A\backslash\left\{ i\right\} \cup\left\{ j\right\} $
if $i\in A$, $j\notin A$, and $A\backslash\left\{ i\right\} \cup\left\{ j\right\} \not \in \F$; and $\mathcal{S}_{ij}\left(A\right)=A$ otherwise. We define $\mathcal{S}_{ij}\left(\mathcal{F}\right)= \{\mathcal{S}_{ij}\left(A\right): A \in \F\}$.

\medskip \noindent $\F$ is called \emph{$n$-compressed} if $A\backslash\left\{ n\right\} \cup\left\{ j\right\} \in\F$
for all $A\in\F$ such that $A\cap\left\{ j,n\right\} =n$, i.e., if $\mathcal{S}_{nj}\left(\mathcal{F}\right) = \F$
for all $j<n$. $\F$ is called \emph{shifted} if $\mathcal{S}_{ij}\left(\mathcal{F}\right) = \F$
for all $j<i$.
\end{definition}

\noindent The following properties of the shifting operator are easy to check.
\begin{claim}
\label{prop:shifting}Let $\mathcal{F}\subset \mathcal{P}\left(\left[n\right]\right)$
be increasing and $t$-intersecting, and let $0 < p < 1$. Then $\s_{ij}\left(\mathcal{F}\right)$ satisfies
the following properties:

\mn \emph{(a)} $\mu_{p}\left(\s_{ij}\left(\mathcal{F}\right)\right) =\mu_{p}\left(\mathcal{F}\right)$,

\mn \emph{(b)} $\mu_{p}\left(\mathcal{F}\backslash\s_{ij}\left(\mathcal{F}\right)\right)\le I_{i}\left(\mathcal{F}\right)$,

\mn \emph{(c)} $I_{i}\left(\s_{ij}\left(\mathcal{F}\right)\right)\le I_{i}(\F)$,
with equality if and only if $\s_{ij}\left(\mathcal{F}\right) = \F$,

\mn \emph{(d)} $\s_{ij}\left(\F\right)$ is increasing and $t$-intersecting.
\end{claim}

\subsubsection{{\it n}-compression by small modifications}
\label{sec:subsub:compression}

The following proposition shows that any increasing $t$-intersecting $\F$ can be transformed into an
$n$-compressed increasing $t$-intersecting $\G$ with $\mu_p(\G) = \mu_p(\F)$ by a sequence of $c$-small modifications,
where $c=I_n(\F)$.
\begin{proposition}\label{cor:shifter}
Let $0 < p < 1$, and let $\F \subset \p([n])$ be an increasing $t$-intersecting family that is not $n$-compressed. Denote $\delta=I_{n}\left(\F\right)$.
Then there exist families $\F=\F_{0},\F_{1},\F_{2},\ldots,\F_{m} \subset \p([n])$, such that $\F_{m}$ is $n$-compressed and
for each $i\in\left[m\right]$ we have:

\mn \emph{(a)} $\mu_{p}\left(\F_{i}\right)=\mu_{p}\left(\F\right)$,

\mn \emph{(b)} $\F_{i}$ is increasing and $t$-intersecting,

\mn \emph{(c)} $\mu_{p}\left(\F_{i-1}\backslash\F_{i}\right)\le\delta$,

\mn \emph{(d)} $I_{n}\left(\F_{i}\right)<\delta$.
\end{proposition}
\begin{proof}
We define $\F_{i}$ inductively. Suppose that $\F_{i-1}$ is not $n$-compressed.
Then for some $j\in\left[n\right]$, we have $\s_{nj}\left(\F_{i-1}\right)\ne\F_{i-1}$.
We choose such a $j$ arbitrarily and define $\F_{i}=\s_{nj}\left(\F_{i-1}\right)$. By
Claim~\ref{prop:shifting}, $\F_{i}$ satisfies the desired properties. Thus,
we only need to show that for some $m\in\mathbb{N}$, $\F_{m}$ is $n$-compressed.
Indeed, by Claim~\ref{prop:shifting}, $I_{n}\left(\F_{i}\right)$ is strictly decreasing (as a function of $i$).
Since all of $\left\{ I_{n}\left(\F_{i}\right)\,:\, i\in\mathbb{N}\right\} $
belong to a finite set of values, this process cannot last forever.
\end{proof}

\subsubsection{Increasing the measure by a small modification}
\label{sec:subsub:increase}

Our next goal is to show that if $\F$ is an $n$-compressed, increasing, $t$-intersecting family, then it
can be transformed to a $t$-intersecting $\G$ with $\mu_p(\G)> \mu_p(\F)$ by a $c$-small modification,
where $c=I_n(\F)$. That is, if $\F$ is already $n$-compressed, then its measure can be \emph{increased}
by a small modification, without sacrificing the $t$-intersection property.

To show this, we need several claims. These claims were proved in~\cite{Filmus13} under the assumption that
$\F$ is shifted, but it turns out that the proof of~\cite{Filmus13} applies also under the weaker assumption
that $\F$ is $n$-compressed. For sake of completeness, we present the claims below.
\begin{lemma}
\label{lem:exact-t}
Let $\F \subset \p([n])$ be a $t$-intersecting $n$-compressed family and let $a,b\in\left[n\right]$. Let
$A \in \F^{(a)}$ and $B \in \F^{(b)}$ be such that $\left|A\cap B\right|=t$ and $n\in A\cap B$.
Then $a+b=n+t$, and $A\cup B=\left[n\right]$.
\end{lemma}
\begin{proof}
It is clearly sufficient to show that $A\cup B=\left[n\right]$. Suppose for a contradiction
that $i\notin A\cup B$. As $\F$ is $n$-compressed, we have $A'=A\backslash\left\{ n\right\}
\cup\left\{ i\right\} \in\F$. However, $\left|A'\cap B\right|=t-1$, contradicting the fact that
$\F$ is $t$-intersecting.
\end{proof}

\noindent The next proposition shows that if an $n$-compressed $t$-intersecting family $\F$ satisfies
$\F \cap \I_n(\F) \nsubseteq \binom{[n]}{(n+t)/2}$, then the measure of $\F$ can be increased by a small
modification.
\begin{proposition}\label{Prop:Increasing1}
Let $0 < p < 1/2$. Let $\F \subset \p([n])$ be an increasing $t$-intersecting $n$-compressed family. Denote $\mathcal{I}_{n}=\mathcal{I}_{n}\left(\F\right)$. Let $a\ne b\in\left[n\right]$ be such that
$a+b=n+t$. Then the families
\[
\G_{1}:=(\F\backslash (\I_n \cap \F)^{(a)}) \cup (\I_n \setminus \F)^{(b-1)} \qquad \mbox{ and } \qquad
\G_{2}:=(\F\backslash (\I_n \cap \F)^{(b)}) \cup (\I_n \setminus \F)^{(a-1)}
\]
are $t$-intersecting, and
\begin{equation}
\mu_{p}\left(\F\right)\le\max\left\{ \mu_{p}\left(\G_{1}\right),\mu_{p}\left(\G_{2}\right)\right\} ,\label{eq:influential sets are wierd1l}
\end{equation}
with equality only if $\G_{1}=\G_{2}=\F$.
\end{proposition}
\begin{proof}
W.l.o.g. we show that $\G_{1}$ is $t$-intersecting. Let $A,B\in\G_{1}$, and suppose for a contradiction
that $\left|A\cap B\right|\le t-1$. Hence, we have either $A \in (\I_n \setminus \F)^{(b-1)}$ or $B \in (\I_n \setminus \F)^{(b-1)}$
or both. Assume, w.l.o.g., $B \in (\I_n \setminus \F)^{(b-1)}$. Then
$B':=B\cup\left\{ n\right\} \in (\I_n \cap \F)^{(b)}$. Note that we have $A':=A\cup\left\{ n\right\} \in\F$. Indeed, either $A \in \F$ and
then $A' \in \F$ since $\F$ is increasing, or $A \in (\I_n \setminus \F)^{(b-1)}$ and then $A' \in (\I_n \cap \F)^{(b)}$. Since $\F$ is
$t$-intersecting, this implies
\begin{equation}
t\le\left|A'\cap B'\right|\le\left|A\cap B\right|+1\le t.\label{eq:a prime cap b prime}
\end{equation}
This allows applying Lemma~\ref{lem:exact-t} to $A',B'$, to get $\left|A'\right|=a$.

Now, as $\left|(A'\backslash\left\{ n\right\}) \cap B'\right|=t-1$ and $\F$ is $t$-intersecting,
we have $A'\backslash\left\{ n\right\} \notin \F$. Hence, $A'\in \left(\I_n \cap \F \right)^{(a)}$,
which yields $A'\notin\G_{1}$. As $A\in\G_{1}$, we must have $A=A'\backslash\left\{ n\right\} \notin\F$.
By the construction of $\G_1$, this means that $A\in \left(\mathcal{I}_{n}\backslash\F \right)^{(b-1)}$,
and therefore, $\left|A'\right|=a=b$, a contradiction.

The proof of~\eqref{eq:influential sets are wierd1l} is a straightforward calculation. Write $\A_{1}=\left(\mathcal{I}_{n}\cap\F \right)^{(a)}$ and $\A_{2}=\left(\mathcal{I}_{n}\cap\F \right)^{(b)}$. Suppose w.l.o.g. that $\mu_{p}\left(\A_{1}\right)\ge\mu_{p}\left(\A_{2}\right)$.
Then
\[
\mu_{p}\left(\G_{2}\right)=\mu_{p}\left(\F\right)-\mu_{p}\left(\A_{2}\right)+\frac{1-p}{p}\mu_{p}\left(\A_{1}\right)>\mu_{p}\left(\F\right),
\]
as asserted. It is also clear that equality can hold only if $\A_1=\A_2 = \emptyset$, that is, if $\G_1=\G_2=\F$.
This completes the proof.
\end{proof}

\noindent The following proposition complements the previous one by showing that the measure of an $n$-compressed
$t$-intersecting $\F$ can be increased by a small modification even if $\F \cap \I_n(\F) \subset \binom{[n]}{(n+t)/2}$.
Recall that $\D_i$ denotes a dictatorship $\D_{i}:=\left\{ A\in \p([n])\,:\, i\in A\right\} $.
\begin{proposition}\label{Prop:Increasing2}
Let $n,t \in \mathbb{N}$ such that $n+t$ is even, and let $0 < p < 1$. Let $\mathcal{F}\subset \p([n])$ be an increasing $n$-compressed $t$-intersecting family such that $I_n(\F)>0$.
Denote $a=\frac{n+t}{2}$. For $1 \leq i \leq n-1$, let
\[
\mathcal{G}_{i}=(\mathcal{F}\backslash \left(\F\cap \I_{n} \cap \D_{i}\right)^{(a)})\cup \left(\mathcal{I}_{n}\backslash \left(\F \cup \D_{i}\right) \right)^{(a-1)}.
\]
Then the families $\mathcal{G}_{i}$ are $t$-intersecting. Moreover, if $0 < p \leq 1/2 - \zeta$, $n > t/(2\zeta)$ and $(\I_n \cap \F)^{(a)} \neq \emptyset$, then
\begin{equation}
\max_{i\in\left[n-1\right]}\left\{ \mu_{p}\left(\mathcal{G}_{i}\right)\right\} >\mu_{p}\left(\mathcal{F}\right).\label{eq:max nu p g i}
\end{equation}
\end{proposition}

\begin{proof}
First we prove that for all $i$, $\G_i$ is $t$-intersecting. Let $A,B\in\mathcal{G}_{i}$, and suppose for a contradiction that
$\left|A\cap B\right|\le t-1$. Denote $A':=A\cup\left\{ n\right\}$ and $B':=B\cup\left\{ n\right\}$, and assume w.l.o.g.
$B\notin\F$, and hence, $\left|B'\right|=a$ and $i\notin B'$.

By the same argument as in the proof of Proposition~\ref{Prop:Increasing1}, we have
$A'\in\F\cap\mathcal{I}_{n}$. On the other hand, by Lemma~\ref{lem:exact-t} (applied for $A',B'$)
we have $A' \cup B' = [n]$, and hence, $i \in A'$ and $|A'|=a$. Thus, $A' \in (\F \cap \mathcal{I}_n \cap \D_i)^{(a)}$,
which implies $A' \not \in \mathcal{G}_i$. This is a contradiction, as $A \in \mathcal{G_i}$ and $\mathcal{G}_i$ is
increasing.

Now we prove Equation~(\ref{eq:max nu p g i}). Note that for any $i \leq n-1$, there is a one-to-one correspondence between
the families $(\I_n \setminus (\F \cup \D_i))^{(a-1)}$ and $(\F \cap \I_n \cap \D_i^c)^{(a)}$. Hence,
\begin{align*}
\mu_{p}\left(\G_{i}\right) & =\mu_{p}\left(\F\right)-\mu_{p}\left(\left(\F\cap\mathcal{I}_{n}\cap\mathcal{D}_{i}\right)^{\left(a\right)}\right)+
\left(\frac{1-p}{p}\right)\mu_{p}\left(\left(\F\cap\mathcal{I}_{n}\cap\mathcal{D}_{i}^{c}\right)^{\left(a\right)}\right)\\
 & =\mu_{p}\left(\F\right)-\mu_{p}\left(\left(\F\cap\mathcal{I}_{n}\right)^{\left(a\right)}\right)+\mu_{p}\left(\left(\F\cap\mathcal{I}_{n}\cap\mathcal{D}_{i}^{c}\right)
 ^{\left(a\right)}\right)+\left(\frac{1-p}{p}\right)\mu_{p}\left(\left(\F\cap\mathcal{I}_{n}\cap\mathcal{D}_{i}^{c}\right)^{\left(a\right)}\right)\\
 & =\mu_{p}\left(\F\right)-\mu_{p}\left(\left(\F\cap\mathcal{I}_{n}\right)^{\left(a\right)}\right)+\frac{1}{p}\mu_{p}\left(\left(\F\cap\mathcal{I}_{n}\cap\mathcal{D}_{i}^{c}\right)
 ^{\left(a\right)}\right).
\end{align*}
Let $\mathcal{K}_{i}:=\F\cap\mathcal{I}_{n}\cap\mathcal{D}_{i}^{c}$. We have
\[
\mathbb{E}_i [\mu_p(\mathcal{K}_i)] = \frac{1}{n-1}\sum_{i=1}^{n-1} \sum_{A\in\mathcal{K}_{i}} \mu_p(A) =
\frac{1}{n-1}\sum_{A \in (\F \cap \I_n)^{(a)}} \sum_{\{i: i \not \in A\} } \mu_p(A) =
\frac{n-a}{n-1} \mu_p((\F \cap \I_n)^{(a)}).
\]
Thus, there exists $i \in [n-1]$ such that $\mu_p(\mathcal{K}_i) \geq \frac{n-a}{n-1} \mu_p((\F \cap \I_n)^{(a)})$. This implies
\[
\max\left\{ \mu_{p}\left(\mathcal{G}_{i}\right)\right\} \ge\mu_{p}\left(\F\right)+\frac{n-a-\left(n-1\right)p}{\left(n-1\right)p}\mu_{p}\left(\left(\F\cap\mathcal{I}_{n}\right)^{(a)}\right) > \mu_p(\F),
\]
where the last inequality holds since $n-a-\left(n-1\right)p>0$ for all $n > t/(2\zeta)$. This completes the proof.
\end{proof}

\noindent Combining Propositions~\ref{Prop:Increasing1} and~\ref{Prop:Increasing2} we obtain:
\begin{corollary}\label{cor:Filmus}
Let $\zeta >0$, let $0 < p \leq 1/2-\zeta$ and let $n \in \mathbb{N}$ with $n >t/(2\zeta)$. Let $\F \subset \p([n])$ be an increasing $n$-compressed $t$-intersecting family that depends on the $n$th coordinate (i.e.,
$I_n(\F)>0$). Then there exists a $t$-intersecting family $\G \subset \p([n])$, such that $\mu_{p}\left(\G\right)>\mu_{p}\left(\F\right)$
and $\mu_{p}\left(\F\backslash\G\right)\le I_{n}\left(\F\right)$.
\end{corollary}

\noindent While $\G$ obtained in Corollary~\ref{cor:Filmus} is $t$-intersecting, it is not necessarily $n$-compressed.
However, by Proposition~\ref{cor:shifter} it can be transformed into an $n$-compressed family $\tilde{\G}$ by a sequence of small
modifications, without decreasing the measure. Then Corollary~\ref{cor:Filmus} can be applied to $\tilde{\G}$ to increase
the measure again. As we show in section~\ref{sec:proof} below, the process can be continued until the $n$th coordinate
becomes non-influential, i.e., the effective number of coordinates decreases. Then one may repeat the whole process
with the $(n-1)$th coordinate etc., so that ultimately, $\F$ can be transformed into a junta by a sequence of small
modifications.

\subsection{Reduction from {\it k}-element sets to the biased measure setting}
\label{sec:sub:reduction}

As shown in several previous works (e.g.,~\cite{DF09,Friedgut08}), EKR-type results for ($t$)-intersecting
subsets of $\binom{[n]}{k}$, \emph{for a sufficiently large $n$}, can be proved by reduction to similar results
on the $\mu_p$ measure of ($t$)-intersecting subsets of $\p([n])$, for an appropriately chosen $p$. In this
subsection we present the lemmas required for performing such a reduction for the stability version of the
Ahlswede-Khachatrian theorem.

The reduction (in our case) works as follows. Let $\F \subset \binom{[n]}{k}$ be a $t$-intersecting family with
$\left|\F\right|> f(n,k,t)-\epsilon\binom{n}{k}$. Recall that $\F^{\uparrow}$ denotes the increasing family generated by
$\F$ (i.e., the minimal increasing family that contains $\F$). We take $p$ slightly larger than $\frac{k}{n}$,
and show that:

\mn (a) $\mu_{p}\left(\F^{\uparrow}\right) \gtrsim \frac{\left|\F\right|}{\binom{n}{k}}>\frac{f(n,k,t)}{\binom{n}{k}}-\epsilon$,

\mn (b) $f(n,p,t) \sim \frac{f(n,k,t)}{\binom{n}{k}}$,

\mn (c) $\mu_{p}\left(\F^{\uparrow}\backslash\e\right) \sim \frac{\left|\F\backslash(\e)^{(k)}\right|}{\binom{n}{k}}$.

%\begin{equation}
%\mu_{p}\left(\F^{\uparrow}\right) \gtrsim \frac{\left|\F\right|}{\binom{n}{k}}>\frac{f(n,k,t)}{\binom{n}{k}}-\epsilon,\label{eq:similar measures}
%\end{equation}
%that
%\begin{equation}
%f(n,p,t) \sim \frac{f(n,k,t)}{\binom{n}{k}},\label{eq: similar juntas}
%\end{equation}
%and that
%\begin{equation}
%\mu_{p}\left(\F^{\uparrow}\backslash\e\right) \sim \frac{\left|\F\backslash(\e)^{(k)}\right|}{\binom{n}{k}}.\label{eq:similar differneces}
%\end{equation}
\mn This essentially reduces stability for subsets of $\binom{[n]}{k}$ to stability in the $\mu_{p}$ setting. We present now three propositions that justify the `$\sim$' in (a)-(c). These propositions, or close variants thereof, were proved in previous works; as they do not appear in the exact form we use them, we present the simple proofs here.

\mn The first proposition, which was essentially proved by Friedgut~\cite{Friedgut08}, shows that~(a) holds.
\begin{proposition}[Friedgut]
\label{prop:monotone-approx}
Let $\delta >0$, let $n,k \in \mathbb{N}$ with $k \leq n$, let $\frac{k+\sqrt{2n\log\left(\frac{1}{\delta}\right)}}{n} \leq p \leq 1$ and let $\F\subset \p\left(\left[n\right]\right)$ be increasing. Then
\[
\mu_{p}\left(\F\right) \geq \frac{\left|\F^{\left(k\right)}\right|}{\binom{n}{k}}\left(1-\delta\right).
\]
\end{proposition}

\noindent Proposition~\ref{prop:monotone-approx} can be proved using the following simple corollary of the local LYM inequality.
\begin{proposition}
\label{Prop:WeakKK}
Let $\F \subset \p([n])$ be an increasing family. For any $1 \leq k \leq m \leq n$, we have
$|\F^{(m)}|/{\binom{n}{m}} \geq |\F^{(k)}|/{\binom{n}{k}}$.
\end{proposition}

\begin{proof}[Proof of Proposition~\ref{prop:monotone-approx}]
By Proposition~\ref{Prop:WeakKK}, we have
\begin{align*}
\mu_{p}\left(\F\right) &\geq \sum_{m=k}^{n}\frac{\left|\F^{\left(m\right)}\right|}{\binom{n}{m}}\mu_{p}\left(\binom{[n]}{m}\right)
\ge\frac{\left|\F^{\left(k\right)}\right|}{\binom{n}{k}}\sum_{m=k}^{n}\mu_{p}\left(\binom{[n]}{m}\right) =
\frac{\left|\F^{\left(k\right)}\right|}{\binom{n}{k}} \mu_p(\{S \subset [n]: |S| \geq k\}) \\
&\ge\frac{\left|\F^{\left(k\right)}\right|}{\binom{n}{k}}\left(1-\delta\right),
\end{align*}
where the last inequality follows from the choice of $p$ by a standard Chernoff bound.
\end{proof}

\noindent The second proposition, proved by Dinur and Safra~\cite{Dinur-Safra}, shows that~(b) holds.
\begin{proposition}[Dinur and Safra]
\label{prop:juntas measure}
Let $j \in \mathbb{N}$, $\zeta,\epsilon\in\left(0,1\right)$, and $p\in [\zeta,1-\zeta]$. There exist
$\delta'=\delta'(\epsilon,j)$ and $n_0=n_0(j,\zeta,\epsilon) \in \mathbb{N}$ such that the following holds for all $n>n_0$.
For any $k\in [\zeta n,(1-\zeta)n] \cap \mathbb{N}$ such that $\left|p-\frac{k}{n}\right|<\delta'$ and for
any $j$-junta $\mathcal{J} \subset \p([n])$, we have
\begin{equation}\label{Eq:Juntas-Measure}
\left|\mu_{p}\left(\mathcal{J}\right)-\frac{\left|\J \cap \binom{[n]}{k} \right|}{\binom{n}{k}}\right|< \epsilon.
\end{equation}
\end{proposition}
\begin{proof}
Assume w.l.o.g. that $\J$ depends only on the coordinates in $[j]$. Since $j$ is fixed, it is sufficient to prove that~\eqref{Eq:Juntas-Measure} holds when $\J = \mathcal{J}_{C}:=\left\{ A \subset [n]:\, A\cap\left[j\right]=C\right\}$ for any $C\subset \left[j\right]$. And indeed,
\begin{align*}
| \mu_{p}(\mathcal{J}) &-\frac{|\{ A\in \binom{[n]}{k} : A\cap [j] = C\}|}{\binom{n}{k}} | =
\left|p^{\left|C\right|}\left(1-p\right)^{j-\left|C\right|}-\frac{\binom{n-j}{k-\left|C\right|}}{\binom{n}{k}}\right| \\
&=\left|p^{\left|C\right|}\left(1-p\right)^{j-\left|C\right|}-
\frac{k \cdot \ldots \cdot \left(k-\left|C\right|+1\right) \cdot \left(n-k\right) \cdot \ldots \cdot \left(n-k+1-j+\left|C\right|\right)}{n\cdot \ldots \cdot \left(n-j+1\right)}\right| \\
&< \left|p^{|C|}\left(1-p\right)^{j-|C|} -\left(\frac{k}{n}\right)^{|C|}\left(\frac{n-k}{n}\right)^{j-|C|}\right| +o_{n \to \infty}\left(1\right) < \epsilon,
\end{align*}
where $n_0,\delta'$ clearly can be chosen such that the last inequality holds.
\end{proof}

\noindent The third proposition, a variant of which was proved by Dinur and Friedgut~\cite{DF09}, shows that~(c) holds.
\begin{proposition}[Dinur and Friedgut]
\label{prop:difference of juntas measure}Let $j \in \mathbb{N}$, $\zeta,\delta'' \in\left(0,1\right),$ and $p\in [\zeta,1-\zeta]$.
There exist $C=C\left(\zeta,j\right)>0$ and $n_{0}\left(t,\zeta,\delta''\right) \in \mathbb{N}$ such that the following holds for all $n>n_0$ and all
$k\in[\zeta n,(1-\zeta)n] \cap \mathbb{N}$ such that $p>\frac{k+\sqrt{2n\log 2}}{n}$. Let $\F \subset \p([n])$ be an
increasing family, and let $\mathcal{J}$ be a $j$-junta such that $\mu_{p}\left(\F\backslash\mathcal{J}\right)<\delta''$. Then
\[
\left|\left(\F\backslash\J\right)^{\left(k\right)}\right|<C\delta''\binom{n}{k}.
\]
\end{proposition}
\begin{proof}
Suppose w.l.o.g. that $\J$ depends on the coordinates in $\left[j\right]$. Since $j$ is fixed, it is sufficient to prove that for any
$E\notin\J$, we have
\[
|\{ A\in\F^{(k)}\,:\, A\cap [j]=E\}| \leq C'\delta''\binom{n}{k},
\]
for some $C'=C'(\zeta,j)>0$. We show that
\begin{equation}\label{Eq:Diff-Measure}
|\{ A\in\F^{(k)}\,:\, A\cap [j]=E\}|<\frac{2\delta''}{p^{\left|E\right|}\left(1-p\right)^{j-\left|E\right|}}\binom{n-j}{k-\left|E\right|},
\end{equation}
which is sufficient, as the right hand side of~\eqref{Eq:Diff-Measure} is $\leq C'\delta''\binom{n}{k}$ by the proof of
Proposition~\ref{prop:juntas measure}. Suppose for a contradiction that~\eqref{Eq:Diff-Measure} fails.
By Proposition \ref{prop:monotone-approx} (with $\delta'=1/2$), we have
\[
\mu_{p}\left(\F_{[j]}^{E}\right)>\frac{2\delta''}{p^{\left|E\right|}\left(1-p\right)^{j-\left|E\right|}} \cdot (1-1/2) =\frac{\delta''}{p^{\left|E\right|}\left(1-p\right)^{j-\left|E\right|}}.
\]
Hence, $\mu_{p}\left(\F\backslash\J\right)\geq p^{\left|E\right|}\left(1-p\right)^{j-\left|E\right|} \mu_{p}\left(\F_{[j]}^{E}\right) >\delta''$, a contradiction. This completes the proof.
\end{proof}

\subsection{Cross-Intersecting Families}
\label{sec:sub:cross-intersecting}

\begin{definition}
Families $\F,\G \subset \p([n])$ are said to be \emph{cross-intersecting} if $A \cap B \neq \emptyset$ for any $A \in \F$ and $B \in \G$.
\end{definition}

The first generalization of the Erd\H{o}s-Ko-Rado theorem to cross-intersecting families was obtained in 1968 by
Kleitman~\cite{Kleitman68}, and since then, many extremal results on cross-intersecting families have been
proved. Such results assert that if $\F,\G$ are cross-intersecting then they cannot be `simultaneously large',
where the latter can be expressed in various ways. We need several such results. The first is a consequence of the Kruskal-Katona theorem \cite{Kruskal63,Katona66}.

\begin{lemma}
\label{lemma:cross-intersecting-slice}Let $n,k,l \in \mathbb{N}$ with $n \geq k+l$, let $r \in \mathbb{N} \cup \{0\}$, and let $\mathcal{A}\subset \binom{[n]}{k},\ \mathcal{B} \subset \binom{[n]}{l}$ be cross-intersecting families. Suppose that $\left|\mathcal{A}\right| \geq \binom{n}{k}-\binom{n-r}{k}$. Then $\left|\mathcal{B}\right|\le\binom{n-r}{l-r}$.
\end{lemma}

\begin{proof}
Let $\bar{\mathcal{A}}:=\left\{[n] \setminus A\,:\, A\in\mathcal{A}\right\}$; then $\left|\bar{\mathcal{A}}\right| \geq \binom{n}{n-k}-\binom{n-r}{n-k-r}$. By the Kruskal-Katona theorem, we have $\left|\partial^{n-k-l}\left(\bar{\mathcal{A}}\right)\right| \geq {n \choose l} - {n-r \choose l-r}$. Since $\mathcal{A}$ and $\mathcal{B}$ are cross-intersecting, we have $\mathcal{B}\cap\partial^{n-k-l}\left(\bar{\mathcal{A}}\right)=\emptyset$. Hence, $\left|\mathcal{B}\right|\le\binom{n}{l}-\left|\partial^{n-k-l}\left(\bar{\mathcal{A}}\right)\right|\le\binom{n-r}{l-r}$, as required.
\end{proof}

Straightforward estimates for the binomial coefficients yield the following consequence.

\begin{lemma}
\label{lem:cross-intersecting}For each $\zeta>0$, there exists $c = c(\zeta)>1$ such that the following holds. Let $\zeta n\leq k_1,k_2\leq(\tfrac{1}{2}-\zeta)n$,
and let $\A \subset \binom{\left[n\right]}{k_1},\ \B \subset \binom{[n]}{k_2}$ be cross-intersecting
families. If $\mu\left(\A\right)>1-\epsilon$, then $\mu\left(\B\right) = O_{\zeta}\left(\epsilon^{c}\right)$.
\end{lemma}

The second extremal result we need is a consequence of Lemma \ref{lemma:mono-decreasing}. It was first proved in~\cite{EKL16+}; we reproduce the proof here, for completeness.
\begin{proposition}[\cite{EKL16+}, Lemma~2.7]
\label{lem:Cross Intersecting} Let $0<p \leq 1/2$, and let $\F,\G \subset \p([n])$ be cross-intersecting families. Then $\mu_{p}\left(\F\right)\le\left(1-\mu_{p}\left(\G\right)\right)^{\log_{1-p}p}$.
\end{proposition}
\begin{proof}
Since $\F$ and $\G$ are cross-intersecting, we have $\F \subset \G^{*}$. Hence, $\mu_{1-p}(\F) \leq \mu_{1-p}(\G^*) = 1-\mu_{1-p}(\bar{\G}) = 1-\mu_p(\G)$. Hence, by Lemma \ref{lemma:mono-decreasing}, $\mu_p(\F) \leq (\mu_{1-p}(\F))^{\log_{1-p}(p)} \leq (1-\mu_p(\G))^{\log_{1-p}(p)}$, as required.
\end{proof}

\noindent Note that equality holds in Proposition \ref{lem:Cross Intersecting} when $\mathcal{F} = \{S \subset [n]:\ B \subset S\} =: \textrm{AND}_{B}$, and $\mathcal{G} = \{S \subset [n]:\ B \cap S \neq \emptyset\} =:\textrm{OR}_{B}$, where $B \subset [n]$.
\newline

\noindent A special case of Proposition \ref{lem:Cross Intersecting}, with a much simpler proof, is as follows.
\begin{lemma}
\label{lem:cross}
If $\F,\G \subset \pn$ are cross-intersecting, then
\[
\mu_{1/2}\left(\F\right)+\mu_{1/2}\left(\G\right)\le1.
\]
\end{lemma}
\begin{proof}
Let $A\sim\mu_{1/2}$. Then $\mu_{1/2}(\F)=\Pr[A\in\F]$. On the other hand, $[n]\setminus A$ is also distributed
according to $\mu_{1/2}$. Therefore, $\mu_{1/2}(\G)=\Pr[[n]\backslash A\in\G]$. Since the families $\F,\G$ are cross-intersecting, the events $\{A\in\F\}$ and $\{[n]\backslash A\in\G\}$ are disjoint. Therefore, we have
$$\mu_{1/2}(\F)+\mu_{1/2}(\G) = \Pr[A\in\F] + \Pr[[n]\backslash A\in\G] = \Pr[\{A \in \F\} \cup \{[n] \backslash A \in \G\}] \leq 1,$$
as required.
\end{proof}

\noindent Before we state the third extremal result, we need a few preliminaries.
\begin{notation}
For $X \subset \mathbb{N}$, $i \in \mathbb{N}$ and $\mathcal{F}\subset \binom{X}{i}$, we write $\mathcal{L}(\mathcal{F})$ for the initial segment of the lexicographic order on $\binom{X}{i}$ with size $|\mathcal{F}|$. We say a family $\mathcal{C} \subset \binom{X}{i}$ is {\em lexicographically ordered} if it is an initial segment of the lexicographic order on $\binom{X}{i}$, i.e., $\mathcal{L}(\mathcal{C})=\mathcal{C}$.
\end{notation}

\noindent The following result was proved by Hilton (see~\cite{Frankl87}, Theorem~1.2).
\begin{proposition}[Hilton]
\label{prop:hilton}
If $\mathcal{F} \subset \binom{[n]}{k}$, $\mathcal{G} \subset \binom{[n]}{l}$
are cross-intersecting, then $\mathcal{L}(\mathcal{F})$, $\mathcal{L}(\mathcal{G})$ are also cross-intersecting.
\end{proposition}

\noindent The third extremal result we need on cross-intersecting families is the following, which was first proved in \cite{EKL16+}; we reproduce the proof for completeness.

\begin{lemma}[\cite{EKL16+}, Lemma~4.7]
\label{lem:Technical-long}
For any $\eta >0$ and any $C \geq 0$, there exists $c_0 = c_0(\eta,C)\in \mathbb{N}$ such that the following holds. Let $n,l,k,d \in \mathbb{N} \cup \{0\}$ with $n\ge (1+\eta)l+k+c_0$ and $l \geq k+c_0-1$. Suppose that $\mathcal{A} \subset \binom{[n]}{l},\, \B \subset \binom{[n]}{k}$
are cross-intersecting, and that
\[
\left|\A\right|\le\left|\mathrm{OR}_{\left[d\right]} \cap \binom{[n]}{l}\right|=\binom{n}{l}-\binom{n-d}{l}.
\]
Then
\[
\left|\mathcal{A}\right|+C\left|\mathcal{B}\right|\le\binom{n}{l}-\binom{n-d}{l}+C\binom{n-d}{k-d}.
\]
\end{lemma}
\begin{proof}
We prove the lemma by induction on $k$. For $k=0$ the lemma holds trivially. Assume now that $k \geq 1$, and that the statement of the lemma holds for $k-1$. For $d=0$, the statement of the lemma holds trivially, so we may assume throughout that $d \geq 1$. By Proposition \ref{prop:hilton}, we may assume that $\mathcal{A}$ and $\mathcal{B}$ are lexicographically
ordered. Since $d \geq 1$, we have $|\mathcal{A}| \leq \binom{n}{l}-\binom{n-1}{l} = \binom{n-1}{l-1}$, so $\mathcal{A} \subset \F_1^{(l)}$, where $\F_{1}^{\left(i\right)}:=\left\{ A\in \binom{[n]}{i}:\,1\in A\right\}$ for each $i\in\left[n\right]$.

We split into two cases: $\A = \F_1^{(l)}$ and $\A \subsetneq \F_1^{(l)}$.

\mn \textbf{Case 1: $\A = \F^{(l)}_{1}$.} First note that $\mathcal{B}\subset \F_{1}^{\left(k\right)}$. Indeed,
suppose on the contrary that $B\in\B$ and $1\notin B$. Since $n \geq k+l$, there exists $A \in \binom{[n]}{l}$ such that $1 \in A$ and $A\cap B=\emptyset$.
Hence, $A\in\F_{1}^{\left(l\right)} = \A$, and $A\cap B=\emptyset$, a contradiction. Hence, we may assume that $\mathcal{B}= \F_{1}^{\left(k\right)}$. We must prove that
\begin{equation} \label{eq:basic-inequality} {n-1 \choose l-1} + C{n-1 \choose k-1} \leq \binom{n}{l}-\binom{n-d}{l}+C\binom{n-d}{k-d}\quad \forall d \geq 1.\end{equation}
This clearly holds (with equality) if $d=1$. To verify it for all $d \geq 2$ it suffices to show that
$${n-1 \choose l-1} + C{n-1 \choose k-1} \leq \binom{n}{l}-\binom{n-2}{l},$$
or equivalently,
$$C{n-1 \choose k-1} \leq \binom{n-2}{l-1}.$$
We have
\begin{align*} \frac{{n-1 \choose k-1}}{\binom{n-2}{l-1}} &= \frac{n-1}{n-k} \frac{{n-2 \choose k-1}}{\binom{n-2}{l-1}}
\leq 2\frac{(l-1)(l-2)\ldots k}{(n-k-1)(n-k-2)\ldots (n-l)}\\
& \leq 2 \left(\frac{l-1}{n-k-1}\right)^{l-k}
 \leq 2 \left(\frac{l-2}{l+\eta l + c -1}\right)^{c-1} \leq \frac{1}{C},
\end{align*}
provided $c_0$ is sufficiently large depending on $\eta$ and $C$, as required.

\mn \textbf{Case 2: $\A\subsetneq\F_1^{(l)}$.} If $\left|\A\right| \leq \binom{n-2}{l-2}$, then
\[
\left|\A\right|+C\left|\B\right|\le\binom{n-2}{l-2}+C\binom{n}{k} \leq \binom{n-1}{l-1} \leq \binom{n}{l}-\binom{n-d}{l}+C\binom{n-d}{k-d},
\]
where the second inequality holds since
\begin{align*}
\frac{{n \choose k}}{{n-1 \choose l-1}-{n-2 \choose l-2}} & = \frac{{n \choose k}}{{n-1 \choose l-2}}
=\frac{n}{n-k} \frac{{n-1 \choose k}}{{n-1 \choose l-2}}
\leq 2 \frac{(l-2)(l-3)\ldots (k+1)}{(n-k-1)(n-k-2)\ldots (n-l+2)}\\
& \leq 2 \left(\frac{l-2}{n-k-1}\right)^{l-k-2}
\leq 2 \left(\frac{l-2}{l+\eta l + c -1}\right)^{c-3}
\leq \frac{1}{C},
\end{align*}
provided $c_0$ is sufficiently large depending on $\eta$ and $C$. Hence, we may assume that
$${n-2 \choose l-2} \leq |\mathcal{A}| \leq {n-1 \choose l-1}.$$
Therefore, since $\A$ is lexicographically ordered, we have $\mathcal{A} \supset \{S \in \binom{[n]}{l}:\ 1,2 \in S\}$. Hence, $B \cap \{1,2\} \neq \emptyset$ for all $B \in \mathcal{B}$. (If there exists $B \in \mathcal{B}$ with $B \cap \{1,2\} = \emptyset$, then since $n \geq k+l$, there exists $A \in \binom{[n]}{l}$ with $A \cap B = \emptyset$ and $1,2 \in A$, but the latter implies $A \in \mathcal{A}$, a contradiction.) Therefore, since $\B$ is lexicographically ordered, we have $\B \supset \F_1^{(k)}$.

Observe that
$$\mathcal{A}_{\left\{ 1,2\right\} }^{\left\{ 1\right\} }\subseteq ([n]\setminus [2])^{(l-1)},\quad \B_{\left\{ 1,2\right\} }^{\left\{ 2\right\} }\subset ([n]\setminus[2])^{(k-1)}$$
are cross-intersecting, and trivially $|\mathcal{A}_{\left\{ 1,2\right\} }^{\left\{ 1\right\} }| \leq {n-2 \choose l-1}$. Hence, by the induction hypothesis (which we may apply since $(n-2) \geq (1+\eta)(l-1) + (k-1) +c_0$ and $l-1 \geq k-1 + c_0-1$, choosing $d=n-2$), we have
\[
\left|\mathcal{A}_{\left\{ 1,2\right\} }^{\left\{ 1\right\} }\right|+C\left|\B_{\left\{ 1,2\right\} }^{\left\{ 2\right\} }\right| \leq {n-2 \choose l-1},\]
and therefore,
\begin{align*} |\A|+C|\B| & = {n-2 \choose l-2} + \left|\mathcal{A}_{\left\{ 1,2\right\} }^{\left\{ 1\right\} }\right| + C {n-1 \choose k-1} + C\left|\B_{\left\{ 1,2\right\} }^{\left\{ 2\right\} }\right| \\
&\leq {n-2 \choose l-2} + {n-2 \choose l-1} + C{n-1 \choose k-1}\\
&= {n-1 \choose l-1} + C{n-1 \choose k-1}\\
& \leq \binom{n}{l}-\binom{n-d}{l}+C\binom{n-d}{k-d},
\end{align*}
using (\ref{eq:basic-inequality}) for the last inequality. This completes the proof.
\end{proof}

We need the following consequence of Lemma \ref{lem:Technical-long}.
\begin{proposition}
\label{lem:from David}Let $n,j,M \in \mathbb{N}$, $\zeta \in\left(0,1/2\right)$, and
let $\zeta n \leq k_{1},k_{2} \leq (\tfrac{1}{2}-\zeta)n$ be such that $\left|k_{2}-k_{1}\right|\le j$. There exists $c=c\left(M,\zeta,j\right) \in \mathbb{N}$ such that the
following holds. Let $\F \subset \binom{[n]}{k_1}$ and $\G\subset \binom{[n]}{k_2}$ be cross-intersecting families
such that for some $d\in\{c,c+1,\ldots,k_2\}$, we have
\begin{equation}\label{Eq:from_David1}
\binom{n-d}{k_{2}-d}\le\left|\G\right|\le\binom{n-c}{k_{2}-c}.
\end{equation}
Then $\left|\F\right|+M\left|\G\right|\le\binom{n}{k_{1}}-\binom{n-d}{k_{1}}+M\binom{n-d}{k_{2}-d}$.
\end{proposition}

\begin{proof}
By Lemma~\ref{prop:hilton}, we may assume that $\mathcal{F}$ and $\mathcal{G}$ are lexicographically
ordered. In addition, by an appropriate choice of $c$, we may assume throughout that $n \geq n_0$ for any $n_0=n_0(M,\zeta,j) \in \mathbb{N}$.

Consider the families $\mathcal{F}_{\left[c\right]}^{\emptyset}$ and $\mathcal{G}_{\left[c\right]}^{\left[c\right]}$ (for $c$ to be specified below), which are clearly cross-intersecting. As $\mathcal{G}$ is lexicographically
ordered, the assumption~\eqref{Eq:from_David1} implies $\mathcal{G} \subset \s_{[c]}$. Moreover, $\mathcal{G}_{\left[c\right]}^{\left[c\right]}$ is also lexicographically ordered, and hence, by~\eqref{Eq:from_David1} we have $\s_{\{c+1,\ldots,d\}} \subset \mathcal{G}_{\left[c\right]}^{\left[c\right]}$. Since
$\mathcal{F}_{\left[c\right]}^{\emptyset}$ cross-intersects $\mathcal{G}_{\left[c\right]}^{\left[c\right]}$, this implies $\mathcal{F}_{\left[c\right]}^{\emptyset} \subset \mathrm{OR}_{\{c+1,\ldots,d\}}$, and thus,
\[
|\mathcal{F}_{\left[c\right]}^{\emptyset}| \leq \binom{n-c}{k_1}-\binom{n-d}{k_1}.
\]
This allows us to apply Lemma~\ref{lem:Technical-long} to the cross-intersecting families
\[
\mathcal{F}_{\left[c\right]}^{\emptyset}\subset \left(\left[n\right]\setminus\left[c\right]\right){}^{(k_{1})},
\quad\mathcal{G}_{\left[c\right]}^{\left[c\right]}\subset
\left(\left[n\right]\backslash\left[c\right]\right)^{\left(k_{2}-c\right)},
\]
with the parameters $n'=n-c$, $l'=k_{1}$, $k'=k_{2}-c$, $d'=d-c$, $C'=M$ and $\eta' = \zeta$, provided that $c:=c_{0}(\zeta,M)+j$, to obtain
\begin{equation}\label{Eq:from_David2}
\left|\mathcal{F}_{\left[c\right]}^{\emptyset}\right|+M\left|\mathcal{G}_{\left[c\right]}^{\left[c\right]}\right|
\le \binom{n-c}{k_{1}}-\binom{n-d}{k_{1}}+M\binom{n-d}{k_{2}-d}.
\end{equation}
(Note that we have $n' \geq (1+\eta')l' + k' + c_0$, provided $n_0(M,\zeta,j)$ is sufficiently large.) Finally, we clearly have
\[
\left|\mathcal{F}\right|\le\binom{n}{k_{1}}-\binom{n-c}{k_{1}}+|\mathcal{F}_{\left[c\right]}^{\emptyset}|,
\]
and the assumption~\eqref{Eq:from_David1} implies $|\mathcal{G}_{\left[c\right]}^{\left[c\right]}|=|\mathcal{G}|$. Therefore,~\eqref{Eq:from_David2} yields
\begin{align*}
\left|\mathcal{F}\right|+M\left|\mathcal{G}\right| & \leq \binom{n}{k_{1}}-\binom{n-c}{k_{1}}+\left|\mathcal{F}_{\left[c\right]}^{\emptyset}\right|+M\left|\mathcal{G}_{\left[c\right]}^{\left[c\right]}\right|\\
 & \le\binom{n}{k_{1}}-\binom{n-c}{k_{1}}+\binom{n-c}{k_{1}}-\binom{n-d}{k_{1}}+M\binom{n-d}{k_{2}-d}\\
 & =\binom{n}{k_{1}}-\binom{n-d}{k_{1}}+M\binom{n-d}{k_{2}-d},
\end{align*}
as asserted.
\end{proof}

\begin{comment}
\noindent We remark that Proposition \ref{lem:from David} is tight for the (uniform) `OR' and `AND' families, i.e. when $\F = \{S \in \binom{[n]}{k_1}:\ S \cap D \neq \emptyset\}$ and $\G = \{S \in \binom{[n]}{k_2}:\ D \subset S\}$, where $D \subset [n]$.
\end{comment}

We also need a `stability result' for cross-intersecting families, giving structural information about one of the families, when the other family is somewhat large. A similar result (with a similar proof) appears in Dinur and Friedgut \cite[Lemma 3.2]{DF09}. One of the main tools is the biased version of Friedgut's Junta Theorem (Theorem \ref{thm:fj}).

\begin{lemma}
\label{lem:dfv}For any $\zeta,\epsilon \in (0,1)$, there exist $s=s(\zeta,\epsilon),\ n_0 = n_{0}(\zeta,\epsilon) \in \mathbb{N}$ such that the following holds. Let $n\ge n_{0}$, let $0 \leq k_1,k_2\leq (\frac{1}{2}-\zeta)n$, let $\f \subset \binom{\left[n\right]}{k_1},\ \g \subset \binom{[n]}{k_2}$ be cross-intersecting, and suppose that $\mu\left(\g\right)\geq\epsilon$. Then there
exists $S \subset [n]$ such that $|S| \leq s$ and $\mu\left(\f_{S}^{\varnothing}\right)<\frac{\epsilon}{2}$.
\end{lemma}
Informally, this lemma says that if we have a pair of cross-intersecting families of uniformity bounded away from $n/2$, and one of the families occupies a positive fraction of its layer, then all but a small number of the sets in the other family have nontrivial intersection with some set of bounded size.
\begin{comment}
Note that Lemma \ref{lem:dfv} implies the $t=1$ case of Lemma \ref{lem:qrc}. Indeed, it follows from Lemma \ref{lem:dfv}
that if $\f$ and $\g$ are $\left(\frac{\epsilon}{2},s\right)$-slice-quasirandom families of measure at least $\epsilon$, then there exist
$A\in\f$ and $B\in\g$ such that $A\cap B=\varnothing$.
\end{comment}

\begin{proof}[Proof of lemma \ref{lem:dfv}]
Clearly, the families $\f^{\uparrow}$ and $\g^{\uparrow}$
are cross-intersecting. Hence, by Lemma \ref{lem:cross}, we have
\[\mu_{1/2}(\f^{\uparrow})+\mu_{1/2}(\g^{\uparrow})\le1.\]
Therefore, using Proposition \ref{prop:monotone-approx}, we have
\[\mu_{1/2}(\f^{\uparrow})\le1-\mu_{1/2}(\g^{\uparrow}) \leq 1-\mu(\g)+\exp(-\zeta^2n/2) \leq 1-\epsilon + \exp(-\zeta^2n/2) \le1-\epsilon/2,\]
provided $n$ is sufficiently large depending on $\zeta$ and $\epsilon$. Let $p=\tfrac{1}{2}(\frac{k_1}{n}+\frac{1}{2})$. The Margulis-Russo Lemma and the Mean Value Inequality imply that there exists $q\in(p,1/2)$ such that
\[
I^{q}(\f^{\uparrow})=\frac{\ud \mu_{q}\left(\f^{\uparrow}\right)}{\ud q}\le \frac{\mu_{1/2}(\f^{\uparrow})-\mu_p(\f^{\uparrow})}{\frac{1}{2}\left(\tfrac{1}{2}-\tfrac{k_1}{n}\right)}\le\frac{2}{\zeta}.
\]
 By Theorem \ref{thm:fj}, there exists $J \subset [n]$ with $|J| = O_{\zeta,\epsilon}(1)$ and a $J$-junta $\langle \h \rangle$ (i.e. $\h \subset \mathcal{P}(J)$), such that $\mu_{q}(\f^{\uparrow} \Delta\left\langle \h\right\rangle)<\frac{\epsilon^{2}}{16}$. Note also that since $q \leq 1/2$, $\f^{\uparrow}$ is monotone, and the function $x \mapsto \mu_x(\mathcal{A})$ is monotone non-decreasing for any increasing family $\mathcal{A} \subset \pn$, we have
\[
\mu_{q}(\f^{\uparrow})\le \mu_{1/2}(\f^{\uparrow}) \leq 1-\epsilon/2.
\]

\begin{claim}
$\mu_{q}\left((\f^{\uparrow})_{J}^{\varnothing}\right)<\frac{\epsilon}{4}.$ \end{claim}
\begin{proof}
Suppose for a contradiction that $\mu_{q}\left((\f^{\uparrow})_{J}^{\varnothing}\right)\ge\frac{\epsilon}{4}$.
Then
\begin{align*}
\mu_{q}\left(\f^{\uparrow}\backslash \langle \h\rangle \right) & =\sum_{B\notin\h}q^{\left|B\right|}\left(1-q\right)^{\left|J\right|-\left|B\right|}\mu_{q}\left((\f^{\uparrow})_{J}^{B}\right)\\
 & \ge\sum_{B\notin\h}q^{\left|B\right|}\left(1-q\right)^{\left|J\right|-\left|B\right|}\frac{\epsilon}{4}\\
 & \ge\frac{\epsilon}{4}\left(1-\mu_{q}\left(\langle \h \rangle \right)\right)\\
 & \ge\frac{\epsilon}{4}\left(1-\mu_{q}\left(\f^{\uparrow}\right)-\mu_{q}\left(\langle \h \rangle \backslash (\f^{\uparrow})\right)\right)\\
 & \ge\frac{\epsilon}{4}\left(\frac{\epsilon}{2}-\frac{\epsilon^{2}}{16}\right)\\
 & >\frac{\epsilon^{2}}{16},
\end{align*}
a contradiction.
\end{proof}
Since $\left(\f_{J}^{\varnothing}\right)^{\uparrow}=\left(\f^{\uparrow}\right)_{J}^{\varnothing}$,
we have
\[
\mu\left(\f_{J}^{\varnothing}\right)\le\mu_{q}\left((\f^{\uparrow})_{J}^{\varnothing}\right)+\exp\left(-\frac{1}{2}\left(\frac{\zeta}{2}-\frac{j}{n}\right)^2n\right) < \frac{\epsilon}{2},
\]
provided $n$ is sufficiently large depending on $\zeta$, $\epsilon$ and $j$, using Proposition \ref{prop:monotone-approx} again. This completes the proof.
\end{proof}

\section{Proof of our stability result for the Ahlswede-Khachatrian theorem}
\label{sec:ak-stability}
In this section, we prove Theorem \ref{thm:main-ak-stability}, our stability result for the Ahlswede-Khachatrian theorem.
\subsection{A Bootstrapping Lemma}
\label{subsection:bootstrapping}

In this subsection, we present a bootstrapping argument showing that if a $t$-intersecting $\F$ is already `somewhat' close to $\e$,
then it must be `very' close to $\e$. We use this argument to show that there exists a `barrier' in the distance of $\F$
from $\e$ that cannot be crossed by performing only small modifications.
\begin{lemma}[Bootstrapping Lemma]\label{Lemma:Main}
Let $t \in \mathbb{N}$ and let $\zeta >0$. Then there exists $C=C(t,\zeta)>0$ such that the following holds. Let $\zeta \leq p \leq 1/2-\zeta$, let $\epsilon >0$, let $\F\subset \p([n])$ be a $t$-intersecting family, and let $\G \in \e$. If
\[
\mu_{p}\left(\mathcal{\F}\cap \G \right)\ge \mu_{p}\left(\G\right)\left(1-\epsilon\right),
\]
then
\[
\mu_{p}\left(\F\backslash \G \right)\le C\epsilon^{\log_{1-p}p}.
\]
\end{lemma}

\begin{proof}
Without loss of generality, we may assume that $\G = \F_{n,t,r}$ for some $r \in \mathbb{N}$. Note that, since $p \leq 1/2-\zeta$, we have $r\le r_0\left(t,\zeta\right)$ (this argument was already
used in the proof of Proposition~\ref{lem:juntas cannot get too close}). Hence, the assumption
$\mu_{p}\left(\mathcal{\F}\cap \F_{n,t,r} \right)\ge \mu_{p}\left(\mathcal{F}_{n,t,r}\right)\left(1-\epsilon\right)$
implies that for any $D\in \F_{2r+t,t,r}$, we have
\begin{equation}\label{Eq:BS1}
\mu_{p}\left(\mathcal{\F}_{\left[2r+t\right]}^{D}\right)\ge 1-O_{t,\zeta}\left(\epsilon\right).
\end{equation}
It is clear that for each $E\notin \F_{2r+t,t,r}$, there exists $D\in \F_{2r+t,t,r}$ such that $\left|D\cap E\right|\le t-1$.
Since $\F$ is $t$-intersecting, for any such $D,E$, the families $\F_{\left[2r+t\right]}^{D}$ and $\F_{\left[2r+t\right]}^{E}$ are
cross-intersecting. By Proposition~\ref{lem:Cross Intersecting} and~\eqref{Eq:BS1}, this implies
\[
\mu_{p}\left(\mathcal{F}_{\left[2t+r\right]}^{E}\right) \le \left(1-\mu_{p}\left(\mathcal{\F}_{\left[2r+t\right]}^{D}\right)\right)^{\log_{1-p}p}
 \le \left(O_{t,\zeta}\left(\epsilon\right)\right)^{\log_{1-p}p}=O_{t,\zeta}\left(\epsilon^{\log_{1-p}p}\right).
\]
Therefore,
\begin{eqnarray*}
\mu_{p}\left(\F\backslash \F_{n,t,r} \right) & = & \sum_{E\notin\F_{2r+t,t,r}}p^{\left|E\right|}\left(1-p\right)^{2r+t-\left|E\right|}\mu_{p}\left(\F_{\left[2r+t\right]}^{E}\right)\\
 & \le & O_{t,\zeta}\left(\epsilon^{\log_{1-p}p}\right)\sum_{E\in\p\left[2r+t\right]}p^{\left|E\right|}\left(1-p\right)^{2r+t-\left|E\right|} = O_{t,\zeta}\left(\epsilon^{\log_{1-p}p}\right).
\end{eqnarray*}
This completes the proof.
\end{proof}

\noindent The following corollary shows that in order to prove Theorem~\ref{thm:biased-ak-stability},
it is sufficient to show that as $\mu_p(\F) \rightarrow f(n,p,t)$, the distance $\mu_p(\F \setminus \e)$
is smaller than a sufficiently small constant $c$.
\begin{corollary}
\label{cor:Bootstrapping}Let $t \in \mathbb{N}$ and $\zeta >0$. There exist positive constants $\tilde{C},c,\epsilon_0$ depending only on $\zeta$ and $t$,
such that for any $t$-intersecting family $\F \subset \p([n])$ and any $p \in [\zeta,1/2-\zeta]$, if
\[
\mu_{p}\left(\F\backslash\e\right)\le c \qquad \mbox{ and } \qquad \mu_{p}\left(\F\right)\ge f(n,p,t)\left(1-\epsilon\right)
\]
for some $\epsilon \leq \epsilon_{0}$, then
\[
\mu_{p}\left(\F\backslash\e\right)\le \tilde{C}\epsilon^{\log_{1-p}p}.
\]
\end{corollary}

\begin{proof}
By the assumption on $\F$, there exists $\G \in \e$ such that
\[\mu_p(\F \cap \G) \geq f(n,p,t)(1-\epsilon)-c.
\]
Hence, by Lemma~\ref{Lemma:Main}, we have $\mu_p\left(\F\backslash\e\right)\le C(\epsilon+c/f(n,p,t))^{\log_{1-p}p}$
for some $C=C(t,\zeta)$. Let $c$ be sufficiently small (as a function of $t,\zeta$) such that $C \cdot (2c/f(n,p,t))^{\log_{1-p}p}
\leq c/2$ for all $p \in [\zeta,1/2-\zeta]$. If $\epsilon>c/f(n,p,t)$, then
\[
\mu_p\left(\F\backslash\e\right) \leq C(\epsilon+c/f(n,p,t))^{\log_{1-p}p} \leq C \cdot (2\epsilon)^{\log_{1-p}p}
\leq \tilde{C} \epsilon^{\log_{1-p}p},
\]
and we are done. Otherwise, we have
\[
\mu_p\left(\F\backslash\e\right) \leq C(\epsilon+c/f(n,p,t))^{\log_{1-p}p} \leq C \cdot (2c/f(n,p,t))^{\log_{1-p}p} < c/2.
\]
In such a case, we can repeat the process with the same $\epsilon$ and $c/2$ instead of $c$. At some stage
$c$ will become sufficiently small so that $\epsilon>c/f(n,p,t)$, and then (as in the first case) we have
$\mu_p\left(\F\backslash\e\right) \leq \tilde{C}\epsilon^{\log_{1-p}p}$, as asserted.
\end{proof}

\noindent Finally, we can use the proof of Corollary~\ref{cor:Bootstrapping} to show the existence of a barrier
that cannot be crossed by small modifications.
\begin{corollary}
\label{cor:Threshold} Let $t \in \mathbb{N}$ and let $\zeta >0$. Let $\F \subset \p([n])$ be $t$-intersecting, and let $p \in [\zeta,1/2-\zeta]$. Let $c,\epsilon_0$ be as in Corollary~\ref{cor:Bootstrapping}, and let
$\epsilon_1:=\min(\epsilon_0,c/f(n,p,t))$. Suppose that
\[
\mu_{p}\left(\F\right)\ge f(n,p,t) \left(1-\epsilon_1\right) \qquad \mbox{ and } \qquad
\mu_{p}\left(\F\backslash\e\right)>c.
\]
Let $\G \subset \p([n])$ be a $t$-intersecting family with $\mu_{p}\left(\G\right)>\mu_{p}\left(\F\right)$ which is
a $(c/2)$-small modification of $\F$ (i.e., $\mu_{p}\left(\F\backslash\G\right)<\frac{c}{2}$). Then
\[
\mu_{p}\left(\G\backslash\e\right)>c.
\]
\end{corollary}
\begin{proof}
Suppose for a contradiction that $\mu_{p}\left(\G\backslash\e\right)\le c$. By the proof of
Corollary~\ref{cor:Bootstrapping}, we have $\mu_{p}\left(\G\backslash\e\right)<\frac{c}{2}$.
This yields
\[
\mu_{p}\left(\F\backslash\e\right)\le\mu_{p}\left(\F\backslash\G\right)+\mu_{p}\left(\G\backslash\e\right)<c,
\]
a contradiction.
\end{proof}

\subsection{Proof of Theorem~\ref{thm:biased-ak-stability}}
\label{sec:proof}

Let us recall the statement of Theorem~\ref{thm:biased-ak-stability}.
\begin{theorem*}
For any $t\in\mathbb{N}$ and any $\zeta>0$, there exists $C=C(t,\zeta)>0$ such that the following holds. Let $p\in\left[\zeta,\frac{1}{2}-\zeta\right]$, and let $\epsilon >0$. If $\F \subset \p([n])$ is a $t$-intersecting family such that $\mu_{p}\left(\F\right)\ge f(n,p,t)\left(1-\epsilon\right)$, then there exists a family $\G$ isomorphic to some $\F_{n,t,r}$, such that $\mu_{p}\left(\F\backslash \G \right)\le
C\epsilon^{\log_{1-p}p}$.
\end{theorem*}

\begin{proof}
Let $c,\epsilon_1$ be as in Corollary~\ref{cor:Threshold}, and let
$\epsilon_2:=\epsilon\left(t,\zeta,\max \left(t+2\ell(t,\zeta),t/(2\zeta),\frac{4}{c^2 p(1-p)}\right)\right)$ in the notations of
Proposition~\ref{lem:juntas cannot get too close}. Define $\epsilon_3 := \min(\epsilon_1,\epsilon_2)$. Let $r: = \ell(t,\zeta)$.

Let $\F \subset \p([n])$ be a $t$-intersecting family. By replacing $\F$ with $\F^{\uparrow}$, we may assume that $\F$ is increasing. We may also assume that
$\mu_{p}\left(\F\right)\ge f(n,p,t)\left(1-\epsilon_3\right)$. (There is no loss of generality in this assumption, as
$C$ can be chosen such that the theorem holds trivially for all $\epsilon>\epsilon_3$.) We would like to show that
$\mu_{p}\left(\F\backslash\e\right)\le c$. This will complete the proof of the theorem by Corollary~\ref{cor:Bootstrapping}.

We let $\F_0=\F$ and construct a sequence $(\F_i)$ of increasing $t$-intersecting families such that each $\F_i$ is obtained from
$\F_{i-1}$ by a series of $c/2$-small modifications. Each `step' in the sequence is composed of compression (using
the process of section~\ref{sec:subsub:compression}) and measure increase (using the process of
section~\ref{sec:subsub:increase}). The construction of $\F_i$ from $\F_{i-1}$ is defined as follows:
\begin{enumerate}
\item If either $\F_{i-1}$ depends on at most $\max\{t+2r,t/(2\zeta)\}$ coordinates, or else
$$\min_{j:\ I_j(\F_{i-1}) > 0} I_j(\F_{i-1}) \geq c/2,$$
then stop.

\item Consider the set of coordinates with non-zero influence on $\F_{i-1}$. Assume w.l.o.g. that this set
is $[m]$, and that $\min_{j \in [m]} I_j(\F_{i-1}) = I_m(\F_{i-1})$. Transform $\F_{i-1}$ to an $m$-compressed increasing $t$-intersecting family $\G_{i-1}$ with $\mu_p(\G_{i-1})=\mu_p(\F_{i-1})$ by a sequence of small modifications
(as described in Proposition~\ref{cor:shifter}).

\item Transform $\G_{i-1}$ into an increasing $t$-intersecting family $\F_i$ with $\mu_p(\F_i)> \mu_p(\G_{i-1})$ by
a small modification (as described in Corollary~\ref{cor:Filmus}, which can be applied since $m > t/(2\zeta)$) and then taking the up-closure to turn the family into an increasing family.
\end{enumerate}
We claim that during all the process, all modifications are $c/2$-small, and that the process terminates after a finite
number of steps.

Indeed, Proposition~\ref{cor:shifter} assures that all modifications in the $m$-compression process are
$I_m(\F_{i-1})$-small, and we have $I_m(\F_{i-1})<c/2$, as otherwise the process terminates by~(1.). Similarly,
Corollary~\ref{cor:Filmus} (which can be applied to $\G_{i-1}$, since $\G_{i-1}$ is $m$-compressed) assures that the
transformation to $\F_i$ is an $I_m(\G_{i-1})$-small modification, and by Proposition~\ref{cor:shifter},
$I_m(\G_{i-1}) \leq I_m(\F_{i-1}) < c/2$. By the construction, the sequence of measures $(\mu_p(\F_i))$ is strictly monotone increasing in $i$. As the measure of a family $\F_i \subset \p([n])$ can assume only a finite number of values, the sequence eventually terminates.

Let $\F_{\ell}$ be the last element of the sequence. As the sequence terminated at the $\ell$'th step, either
$\F_{\ell}$ depends on at most $\max\{t+2r,t/(2\zeta)\}$ coordinates or else $\min_{j:\ I_j(\F_{\ell}) \neq 0} I_j(\F_{\ell}) \geq c/2$. In the latter case,
by Proposition~\ref{Prop:non-juntas have small minimal influence} $\F_{\ell}$ depends on at most
$\frac{4}{c^2 p(1-p)}$ coordinates. Thus, in either case $\F_{\ell}$ depends on at most
$\max \left\{t+2r,t/(2\zeta),\frac{4}{c^2 p(1-p)}\right\}$ coordinates. Since
$\mu_p(\F_{\ell}) > \mu_p(\F) \geq f(n,p,t)\left(1-\epsilon_3\right)$,  Proposition~\ref{lem:juntas cannot get too close}
implies that $\F_{\ell} \in \e$. In particular, $\mu_p(\F_{\ell} \setminus \e) =0 < c$.

Now, we unroll the steps of the sequence. As $\F_{\ell}$ was obtained from $\F_{\ell-1}$ by a $c/2$-small modification,
Corollary~\ref{cor:Threshold} implies $\mu_p(\F_{\ell-1} \setminus \e) < c$. The same holds for any step of the
sequence, and thus, by (reverse) induction, we get $\mu_p(\F \setminus \e) = \mu_p(\F_{0} \setminus \e) < c$. As
mentioned above, this completes the proof of the theorem by Corollary~\ref{cor:Bootstrapping}.
\end{proof}

\subsection{Tightness of Theorem~\ref{thm:biased-ak-stability}}
\label{sec:sub:tightness}

As mentioned in the introduction, Theorem~\ref{thm:biased-ak-stability} is tight (up to a factor depending upon $t$ and $\zeta$ alone) for the families
\begin{align*}
\tilde{\mathcal{H}}_{n,t,r,s} & :=\left\{ A\subset \mathcal{P}\left(\left[n\right]\right)\,:\,|A \cap [t+2r]| \geq t+r,\, A \cap \left\{t+2r+1,\ldots,t+2r+s\right\} \ne\emptyset\right\} \\
 & \cup\left\{ A\subset \mathcal{P}\left(\left[n\right]\right)\,:\,|A \cap [t+2r]|=t+r-1,\,\left\{t+2r+1,\ldots,t+2r+s\right\}\subset A\right\},
\end{align*}
for all sufficiently large $n$ and $s$. Here is the computation showing this. Let $\zeta \leq p \leq 1/2-\zeta$. Choose $r \in \mathbb{N} \cup \{0\}$ such that $p\in\left[\frac{r}{t+2r-1},\frac{r+1}{t+2r+1}\right]$; then we have
\begin{equation}\label{eq:biased-1}
f(n,p,t) = \mu_p(\F_{n,t,r}) = \sum_{i=t+r}^{t+2r}\binom{t+2r}{i}p^i(1-p)^{t+2r-i}
\end{equation}
for all $n \geq t+2r$. For all $n \geq t+2r+s$, we have
\begin{equation}\label{eq:biased-2}
\mu_{p}\left(\tilde{\mathcal{H}}_{n,t,r,s}\right)= f(n,p,t)\left(1-\left(1-p\right)^{s}\right)+ {{t+2r}\choose{t+r-1}} p^{t+r-1}\left(1-p\right)^{r+1}p^{s},
\end{equation}
and
\begin{equation}
\label{eq:biased-3}
\mu_{p}\left(\tilde{\mathcal{H}}_{n,t,r,s}\backslash\mathcal{F}_{n,t,r}\right)={{t+2r}\choose{t+r-1}} p^{t+r-1}\left(1-p\right)^{r+1}p^{s}.
\end{equation}
Note that all the expressions (\ref{eq:biased-1}), (\ref{eq:biased-2}) and (\ref{eq:biased-3}) are independent of $n$. Moreover, since $\zeta < p <1/2-\zeta$, we have $r = O_{t,\zeta}(1)$ and therefore ${{t+2r}\choose{t+r-1}} p^{t+r-1}\left(1-p\right)^{r+1} = \Theta_{t,\zeta}(1)$. Hence, for all $s \geq s_0(t,\zeta)$ and all $n \geq t+2r+s$, we have
\[
\mu_{p}\left(\tilde{\mathcal{H}}_{n,t,r,s}\right)= f(n,p,t)\left(1-\left(1-p\right)^{s}\right) + \Theta_{t,\zeta}(1)p^s \geq
f(n,p,t)\left(1-\tfrac{1}{2}\left(1-p\right)^{s}\right),
\]
while $\mu_{p}(\tilde{\mathcal{H}}_{n,t,r,s}\backslash\mathcal{F}_{n,t,r}) = \Theta_{t,\zeta}(1)p^s$. Writing $\epsilon:=\tfrac{1}{2}(1-p)^s$, we
have
$$\mu_{p}(\tilde{\mathcal{H}}_{n,t,r,s}) \geq f(n,p,t) \left(1-\epsilon\right),\quad \mu_{p}(\tilde{\mathcal{H}}_{n,t,r,s}\backslash\mathcal{F}_{n,t,r}) = \Theta_{t,\zeta}( \epsilon^{\log_{1-p}p}),$$
which is tight for Theorem \ref{thm:biased-ak-stability}.

We remark that it is very easy to see that the families $\mathcal{H}_{n,k,t,r,d}$ (defined in the Introduction) are tight for Theorem \ref{thm:main-ak-stability}, for $n$ and $d$ sufficiently large.

\begin{comment}
We may now deduce that Theorem \ref{thm:main-ak-stability} is tight (again, up to a factor depending upon $t$ and $\zeta$ alone) for the families
\begin{align*}
\mathcal{G}_{n,k,t,r,s} & :=\left\{ A \in \binom{[n]}{k}\,:\,|A \cap [t+2r]| \geq t+r,\, A \cap \left\{t+2r+1,\ldots,t+2r+s\right\} \ne\emptyset\right\} \\
 & \cup\left\{ A \in \binom{[n]}{k}\,:\,|A \cap [t+2r]|=t+r-1,\,\left\{t+2r+1,\ldots,t+2r+s\right\}\subset A\right\},
\end{align*}
for sufficiently large $n$ and $s$. Indeed, note that $\tilde{\F}_{n,t,r}$ and $\tilde{\G}_{n,t,r,s}$ are both juntas depending only upon the coordinates in $[t+2r]$ and $[t+2r+s]$, respectively.

For any family $\F \subset \p([m])$ and any $N>m$ and $k \in [N]$, let us define $\F_{N,k} = \{S \in \binom{[N]}{k}:\ S \cap [m] \in \F\}$. It is easy to check that
$$\lim_{N \to \infty} \frac{|\F_{N,\lfloor pN \rfloor}|}{\binom{N}{\lfloor pN\rfloor}} = \mu_p(\F).$$
\end{comment}

\subsection{Proof of Theorem~\ref{thm:main-ak-stability}}
\label{subsection:uniform}

In this section we present the proof of Theorem~\ref{thm:main-ak-stability}. First, we deduce a `weak stability' result
from Theorem~\ref{thm:biased-ak-stability}, using the reduction technique presented in section~\ref{sec:sub:reduction}. Then,
we use a `bootstrapping' technique similar to in the proof of Lemma~\ref{Lemma:Main}, to leverage the weak stability result into the assertion of the theorem.
\begin{proposition}[Weak stability theorem]\label{Thm:Large-n}
Let $t\in\mathbb{N}$ and $\zeta,\epsilon>0$. There exist $C=C\left(t,\zeta\right)>0$ and $n_{0}\left(t,\zeta,\epsilon\right) \in \mathbb{N}$ such that
the following holds for all $n>n_0$, all $k\in [\zeta n,(1/2-\zeta)n] \cap \mathbb{N}$ and $p=\frac{k+\sqrt{4n\log n}}{n}$.
Let $\F \subset \binom{[n]}{k}$ be a $t$-intersecting family that satisfies
$|\F| \ge f(n,k,t)-\epsilon \binom{n}{k}$. Then there exists $\G$ isomorphic to some $\F_{n,k,t,r}$ such that
$|\F\backslash \G| \le C\epsilon^{\log_{1-p}p} {{n}\choose{k}}$, where $r \leq C$.
\end{proposition}

\begin{proof}
Let $\F,k$ and $p$ satisfy the assumption of the proposition, and let
$\F^{\uparrow} \subset \p([n])$ be the increasing family generated by $\F$. By Proposition~\ref{prop:monotone-approx},
for a sufficiently large $n$,
\[
\mu_{p}\left(\F^{\uparrow}\right)\ge \left(\frac{f(n,k,t)}{\binom{n}{k}} - \epsilon \right) \left(1-\frac{1}{n} \right) \geq
\frac{f(n,k,t)}{\binom{n}{k}} - 2\epsilon.
\]
Let $\F_{n,p,t,r} \in\e$ be a family for which the maximal $\mu_p$ measure is attained, i.e., $\mu_{p}\left(\F_{n,p,t,r}\right)=f(n,p,t)$.
As $p \leq 1/2-\zeta/2$ (which holds assuming that $n$ is sufficiently large), $\F_{n,p,t,r}$ depends on at most $j$ coordinates, for some $j=j(t,\zeta) \in \mathbb{N}$. Hence, by Proposition~\ref{prop:juntas measure},
\[
f(n,p,t)=\mu_{p}\left(\F_{n,p,t,r}\right) < \frac{|\F_{n,p,t,r}^{(k)}|}{\binom{n}{k}}+\epsilon \leq \frac{f(n,k,t)}{\binom{n}{k}} + \epsilon.
\]
Thus, $\mu_{p}\left(\F^{\uparrow}\right) \ge f(n,p,t)-3\epsilon$. Since $\F^{\uparrow}$ is $t$-intersecting, we can apply to it
Theorem~\ref{thm:biased-ak-stability} to get
\[
\mu_p(\F^{\uparrow} \setminus \tilde{\G}) \leq C' \epsilon^{\log_p (1-p)},
\]
for some $\tilde{\G} \in \e$ and $C'=C'(t,\zeta)$. Finally, denoting $\G := \tilde{\G}^{(k)}$, we obtain by Proposition~\ref{prop:difference of juntas measure}
\[
|\F\backslash \G| = |(\F^{\uparrow} \setminus \tilde{\G})^{(k)}| < C \epsilon^{\log_{1-p}p} {{n}\choose{k}},
\]
for a sufficiently large $C=C(t,\zeta)$, as asserted.
\end{proof}

\mn Proposition~\ref{Thm:Large-n} shows that the assertion of Theorem~\ref{thm:main-ak-stability} holds for all $n \geq n_0(t,\zeta,\epsilon)$.
This is not sufficient for Theorem~\ref{thm:main-ak-stability}, in proving which we may only assume $n$ to be large in terms of $t,\zeta$ (and not in terms of $\epsilon$). However,
we can apply Proposition~\ref{Thm:Large-n} with any \emph{moderately small} $\epsilon_0(t,\zeta)>0$ to conclude that for any $n \geq n_1(t,\zeta)$, if $\F$ satisfies the assumption
of Theorem~\ref{thm:main-ak-stability} then there exists $\G \cong \F_{n,k,t,r}$ such that $|\F\backslash \G| \le \epsilon_0 {{n}\choose{k}}$.
In the proof of Theorem~\ref{thm:main-ak-stability} below, we use this weak stability version, with $\epsilon_0$ chosen in such a way that we will be able to use Proposition~\ref{lem:from David} to bootstrap the `weak stability' into `strong stability'.

\mn Let us recall the formulation of Theorem~\ref{thm:main-ak-stability}.
\begin{theorem*}
Let $n,t,d\in\mathbb{N}$, $\zeta \in (0,1/2)$, and $k\in\left(\zeta n,\left(\frac{1}{2}-\zeta\right)n\right)$.
There exists $C=C\left(t,\zeta\right) >0$ such the following holds. Let $\F \subset \binom{[n]}{k}$ be a $t$-intersecting
family with $|\F|> f(n,k,t)-\frac{1}{C}\binom{n-d}{k}$. Then there exists $\G$ isomorphic to some $\F_{n,k,t,r}$
such that $\left|\F\backslash\G\right|<C\binom{n-d}{k-d}$, where $r \leq C$.

%\dnote{The last inequality of the proof only holds if $n-d-k = \Theta(n)$. Alternatively we can replace the hypothesis $|\F|> f(n,k,t)-\frac{1}{C}\binom{n-d}{k}$ with the hypothesis $|\F|> f(n,k,t)-\frac{1}{C}\binom{n-d-c'}{k}$, where $c'$ is an integer depending on $\eta$ and $t$ alone. (If $C$ is required to be an integer then you can take $c'=C$ of course.)}
\end{theorem*}

\begin{proof}[Proof of Theorem~\ref{thm:main-ak-stability}]
Recall that for fixed $t,\zeta$, all elements of $\e$ are juntas on at most $j=j(t,\zeta)$ elements.
Denote $c=c\left(2^{j},t,\zeta\right)$ in the notations of Proposition~\ref{lem:from David}.
Let $\F$ be a family that satisfies the assumption of the theorem
(with a sufficiently large $C=C(t,\zeta)>0$ to be specified below). Clearly, we may assume that $d \leq k+1$. By increasing $C$ if necessary, we may assume that $d \geq d_0(t,\zeta)$ for any $d_0(t,\zeta) \in \mathbb{N}$ and that $n \geq n_0(t,\zeta)$ for any $n_0(t,\zeta) \in \mathbb{N}$.

Provided $n_0 = n_0(t,\zeta)$ is sufficiently large, we have $\binom{n-c-j}{k-c-j} = \Theta_{t,\zeta} \left(\binom{n}{k} \right)$.
Hence, we can apply Proposition~\ref{Thm:Large-n} to conclude that there exists $\G\in (\e)^{(k)}$ such that
\begin{equation}\label{Eq:MainThm1}
\left|\F\backslash\G\right|<\binom{n-c-j}{k-c-j}.
\end{equation}
Suppose w.l.o.g. that $\G$ depends only on the coordinates in $\left[j\right]$, and denote by
$\G'$ the restriction of $\G$ to $\mathcal{P}\left(\left[j\right]\right)$ (i.e., $\G'=\{A \cap [j]:\ A \in \G\}$).
Let $O\notin\G'$ be such that $\left|\F_{\left[j\right]}^{O}\right|$
is maximal. We would like to show that
$\left|\F_{\left[j\right]}^{O}\right|\le \binom{n-d}{k-d}$. This will complete the proof, as
\[
|\F \setminus \G| = \sum_{S \in \p([j]) \setminus \G'} \left|\F_{\left[j\right]}^{S}\right| \leq 2^j \left|\F_{\left[j\right]}^{O}\right|.
\]
Suppose for a contradiction that
\begin{equation}\label{Eq:MainThm2}
\left|\F_{\left[j\right]}^{O}\right|> \binom{n-d}{k-d}.
\end{equation}
%As $\G'$ is a maximal (w.r.t. inclusion) $t$-intersecting subset of $\p([j])$ and $O \notin \G'$,
It is clear that there exists $I\in\G'$ such that $\left|I\cap O\right|\le t-1$. We have
\begin{equation}\label{Eq:MainThm3}
\left|\F\right| = \sum_{S \in \p([j])} |\F_{[j]}^S| \le \left|\F_{\left[j\right]}^{I}\right| + 2^{j}\left|\F_{\left[j\right]}^{O}\right| + \left(|\G|-\binom{n-j}{k-\left|I\right|} \right)
\end{equation}
(where the two last summands are upper bounds on $\sum_{S \in \p([j]) \setminus \G'} |\F_{[j]}^S|$ and
$\sum_{S \in \G' \setminus \{I\}} |\F_{[j]}^S|$, respectively). Now, we note that since $\F$ is $t$-intersecting, the families
$\F_{\left[j\right]}^{I}$, $\F_{\left[j\right]}^{O}$ are cross-intersecting. We have
\[
\left|\F_{\left[j\right]}^{O}\right|\le\binom{n-c-j}{k-c-j}\le\binom{n-j-c}{k-\left|O\right|-c}
\]
(where the first inequality follows from~\eqref{Eq:MainThm1} and the second holds provided $n_0=n_0(t,\zeta)$ is sufficiently large), and
on the other hand,
\[
\left|\F_{\left[j\right]}^{O}\right| \geq \binom{n-d}{k-d} \geq \binom{n-j-d}{k-\left|O\right|-d}
\]
(where the first inequality follows from~\eqref{Eq:MainThm2} and the second holds trivially). Thus, we can apply Proposition~\ref{lem:from David} to get
\[
\left|\F_{\left[j\right]}^{I}\right| + 2^{j}\left|\F_{\left[j\right]}^{O}\right| \leq
\binom{n-j}{k-\left|I\right|}-\binom{n-j-d}{k-\left|I\right|} + 2^{j} \binom{n-j-d}{k-\left|O\right|-d}.
\]
By~\eqref{Eq:MainThm3}, this implies
\begin{align*}
\left|\F\right| &\leq \binom{n-j}{k-\left|I\right|}-\binom{n-j-d}{k-\left|I\right|} + 2^j \binom{n-j-d}{k-\left|O\right|-d}+
\left(|\G|-\binom{n-j}{k-\left|I\right|} \right) \\
& = |\G|-\binom{n-j-d}{k-\left|I\right|} + 2^j \binom{n-j-d}{k-\left|O\right|-d}\\
&\leq f(n,k,t) - \frac{1}{C} \binom{n-d}{k},
\end{align*}
where the last inequality holds for all $d_0 \leq d \leq k+1$ and all $n \geq n_0$, provided $C=C(t,\zeta)$, $n_0 = n_0(t,\zeta)$ and $d_0 = d_0(t,\zeta)$ are all sufficiently large. This
contradicts our assumption on $\F$, completing the proof.
\end{proof}

\section{Proof of our weak regularity lemma}
\label{sec:reg}
In this section we prove Theorem \ref{thm:weak-reg}, our `weak regularity lemma' for hypergraphs of uniformity linear in the number of vertices. First, we need some preliminaries.

Jensen's inequality states that for any convex function $f\colon\mathbb{R}\to\mathbb{R}$, and for any real-valued, integrable random variable $X$, we have
\begin{equation}
\label{eq:jensen}
\mathbb{E}[f(X)]\ge f(\mathbb{E}[X]).
\end{equation}
It turns out that under certain conditions, if the inequality (\ref{eq:jensen}) is approximately an equality, then $X$ is `highly concentrated' around its mean. The following is a restatement of Lemma 7 of Fox \cite{fox}; it may be seen as a stability version of Jensen's inequality for the function $x \mapsto x \log x$.
\begin{lemma}
\label{lem:fox}
Let $\Omega$ be a finite probability space, let $X\colon\Omega\to\mathbb{R}_{\ge0}$
be a nonnegative random variable, and let $f\colon\mathbb{R}_{\ge0}\to\mathbb{R}$
be the convex function defined by
$$f\left(x\right):=\begin{cases}
x\log x & x > 0\\
0 & x=0,
\end{cases}$$
and let $\beta\in\left(0,1\right)$. Then
\[\mathbb{E}\left(f\left(X\right)\right) \ge f\left(\mathbb{E}\left(X\right)\right) +\left(1-\beta+f\left(\beta\right)\right)\Pr\left[X\leq \beta\mathbb{E}\left(X\right)\right]\mathbb{E}\left(X\right).\]
\end{lemma}

\begin{comment}
We will be interested in the setting where $\beta$ is bounded away from $1$, and where $\mathbb{E}\left[X\right]$ is bounded from above. In this setting, Lemma \ref{lem:fox} says that
\[
\mathbb{E}\left[f\left(X\right)\right]\ge f\left(\mathbb{E}\left[X\right]\right)+\Omega\left(\Pr\left[X\leq \beta\mathbb{E}\left(X\right)\right]\right).
\]
In other words, near-equality holds in (\ref{eq:Jensen}) for the
function $f$ only if the random variable $X$ is highly concentrated
around its mean.
\end{comment}

The following is an easy corollary of Lemma \ref{lem:fox}.
\begin{corollary}
\label{cor:foxconn} Let $f$ be the function in Lemma \ref{lem:fox}. For each $\lambda,\delta\in\left(0,1\right)$ and $C>0$,
there exists $\eta=\eta\left(\lambda,\delta,C\right)>0$ such that
the following holds. Let $\Omega$ be finite probability space, and
suppose that $\Pr(\omega) \geq \lambda$ for all $\omega \in\Omega$. Let $X\colon\Omega\to\mathbb{R}_{\ge0}$
be a nonnegative random variable such that $\mathbb{E}[X] \leq C$ and
\[
\mathbb{E}\left[f\left(X\right)\right]<f\left(\mathbb{E}[X]\right)+\eta.
\]
Then
\[
\|X-\mathbb{E}[X]\|_{\infty} < \delta.
\]
\end{corollary}
Corollary \ref{cor:foxconn} says that if Jensen's inequality (\ref{eq:jensen}) is close to being an equality for the function $f$ and a `well-behaved' random variable $X$, then the random variable $X$ is highly concentrated around its mean.
\begin{proof}[Proof of Corollary \ref{cor:foxconn}]
Let $\eta=\eta\left(\lambda,\delta,C\right)>0$ to be chosen later. Suppose that
\[
\mathbb{E}\left[f\left(X\right)\right]<f\left(\mathbb{E}[X]\right)+\eta,
\]
and suppose for a contradiction that $\|X-\mathbb{E}[X]\|_{\infty} \geq \delta$. First suppose that $\min_{\omega \in \Omega}X(\omega) \leq \mathbb{E}[X] - \delta$. Then
$$\Pr\left[X \leq \left(1-\frac{\delta}{\mathbb{E}[X]}\right)\mathbb{E}[X]\right] \geq \lambda,$$
so by applying Lemma \ref{lem:fox} with $\beta = 1-\delta/\mathbb{E}[X]$, we have
$$\mathbb{E}[f(X)] \geq f(\mathbb{E}[X]) + \left( \frac{\delta}{\mathbb{E}[X]} + f\left(1-\frac{\delta}{\mathbb{E}[X]}\right)\right) \lambda \mathbb{E}[X].$$
Let $\gamma := \delta/\mathbb{E}[X]$; then $\gamma \in [0,1]$. It is easily checked that
$$\gamma+f(1-\gamma) \geq \gamma^2/2 \quad \forall \gamma \in [0,1].$$
It follows that
$$\mathbb{E}[f(X)] \geq f(\mathbb{E}[X]) + \frac{\lambda \delta^2}{2 \mathbb{E}[X]} \geq \frac{\lambda \delta^2}{2C}.$$

Second, suppose that $\min_{\omega \in \Omega}X(\omega) > \mathbb{E}[X] - \delta$; then $\max_{\omega \in \Omega}X(\omega) \geq \mathbb{E}[X] + \delta$. Let $M = \min_{\omega \in \Omega}X(\omega)$; then
$$(1-\lambda)M + \lambda (\mathbb{E}[X]+\delta) \leq \mathbb{E}[X],$$
so
$$M \leq \mathbb{E}[X] - \frac{\lambda \delta}{1-\lambda}.$$
Hence, by the argument above, replacing $\delta$ with $\lambda \delta /(1-\lambda)$, it follows that
$$\mathbb{E}[f(X)] \geq f(\mathbb{E}[X]) + \frac{\lambda^3 \delta^2}{2(1-\lambda)^2 C}.$$
Choosing
$$\eta = \min\left\{\frac{\lambda \delta^2}{2C},\frac{\lambda^3 \delta^2}{2(1-\lambda)^2C}\right\}$$
yields a contradiction.
\end{proof}

\subsection*{A `potential' argument}
The idea of the proof of Theorem \ref{thm:weak-reg} is to define a non-positive potential function $\phi:\ \p\left(\binom{[n]}{k}\right) \times \p([n]) \to [-1/e,0]$ with the following properties.
\begin{enumerate}
\item $\phi\left(\f,S\right)=0$ if and only if $\f$ is an $S$-junta.
\item If $S\subset S'\subset \left[n\right]$, then $\phi\left(\f,S\right)\le\phi\left(\f,S'\right)$.
\item If $k/n$ is bounded away from $0$ and $1$, then for any $\F \subset \binom{[n]}{k}$ and any $S\subset \left[n\right]$, {\em either} there exist sets $B_{1},\ldots B_{l}\subset S$ such that the junta $\left\langle \{ B_{1},\ldots,B_{l}\}\right\rangle $
satisfies the conclusion of Theorem \ref{thm:weak-reg}, {\em or}
there exists a set $S' \supset S$ that is not much larger than $S$, and such that $\phi\left(\f,S'\right)$
is significantly larger than $\phi\left(\f,S\right)$. In the latter case, we replace $S$ by $S'$ and repeat; since $\phi$ is bounded from above by $0$, the former case must occur after a bounded number of steps.
\end{enumerate}
We now proceed to define our potential function $\phi$.

\begin{definition}[The $\left(n,k,J\right)$-biased distribution]
For each $J \subset \left[n\right]$, we define the {\em $\left(n,k,J\right)$-biased distribution} on $\p\left(J\right)$ by
\[
\mu_{\left(n,k,J\right)}\left(B\right)=\Pr_{A\sim\binom{\left[n\right]}{k}}\left[A\cap J=B\right] \quad \forall B \subset J,
\]
where $A \sim \binom{[n]}{k}$ denotes a uniform random element of $\binom{[n]}{k}$. We write $B\sim\mu_{(n,k,J)}$ if $B$ is chosen according this distribution.

For $\f\subset \binom{\left[n\right]}{k}$ and $J\subset \left[n\right]$, we define $\alpha^B_{J}:=\mu\left(\f_{J}^{B}\right)$ for each $B \subset J$. We define our potential function $\phi:\ \p\left(\binom{[n]}{k}\right) \times \p([n])$ by
\begin{align*}
\phi\left(\f,J\right) & =\underset{B\sim\mu_{\left(n,k,J\right)}}{\mathbb{E}}\alpha^B_{J}\log\alpha^B_{J} \quad \forall \f \subset \binom{[n]}{k},\ J \subset [n].
\end{align*}
\end{definition}

Since $-1/e \leq f(x) \leq 0$ for all $x \in [0,1]$, we have $-1/e \leq \phi(\f,J) \leq 0$ for all $\f \subset \binom{[n]}{k}$ and all $J \subset [n]$; we have $\phi\left(\f,J\right)=0$ if and only if $\f$ is a $J$-junta. Moreover, it follows from Jensen's inequality that $\phi\left(\f,S\right)\le\phi\left(\f,S'\right)$ for any $S\subset S'$. We note the similarity between the definition of $\phi$ and the definition of the entropy of a random variable. We note also that the product space analogue of the function $\phi$ (in the slightly simpler setting where $\mu_{(n,k,J)}$ is replaced by a product distribution on $\p(J)$) was considered by Friedgut and Regev in \cite{fr}, and used in a similar way to in the sequel.

\begin{definition}
We say that a family $\f\subset \binom{\left[n\right]}{k}$ is {\em $\left(\eta,h\right)$-potentially
stable} if $\phi\left(\f,J\right) < \phi\left(\f,\varnothing\right)+\eta$ for all sets $J \subset [n]$ with $|J| \leq h$.
\end{definition}

We recall from the Introduction the definition of slice-quasirandomness.

\begin{definition}
If $\eta >0$ and $h \in \mathbb{N}$, we say that a family $\F \subset \binom{[n]}{k}$ is {\em $(\delta,h)$-slice-quasirandom} if for any $J \subset [n]$ with $|J| \leq h$, and any $B \subset J$, we have $|\mu(\F_J^B) - \mu(\F)| < \delta$.
\end{definition}

By virtue of Corollary \ref{cor:foxconn}, there is a close connection between potential stability and slice-quasirandomness. Indeed, the following lemma says that if $k/n$ is bounded away from $0$ and $1$, then $(\eta,h)$-potential stability implies $(\delta,h)$-slice-quasirandomness, provided $\eta$ is sufficiently small.

\begin{lemma}
\label{lem:potential-qr}
For any $\zeta,\delta >0,\ h\in\mathbb{N}$
there exist $\eta=\eta\left(\delta,h,\zeta\right)>0$ and
$n_{0}=n_{0}\left(\zeta,h\right) \in \mathbb{N}$ such that the following holds. Let
$n\ge n_{0}$, let $\zeta n\le k\le\left(1-\zeta\right)n$, and let
$\f\subset \binom{\left[n\right]}{k}$ be $\left(\eta,h\right)$-potentially
stable. Then $\f$ is $\left(\delta,h\right)$-slice-quasirandom.
\end{lemma}

\begin{proof}
Let $\f$ be as in the statement of the lemma, and let $\eta = \eta(\delta,h,\zeta) >0$ to be chosen later. Let $J \subset [n]$ with $|J| \leq h$. Let $\Omega = \p(J)$, and equip  $\Omega$ with the probability distribution $\mu_{(n,k,J)}$. Since $\zeta n \leq k \leq (1-\zeta)n$, we have
$$\Pr(\omega) \geq \frac{\binom{n-h}{\lceil \zeta n \rceil-h}}{\binom{n}{\lceil \zeta n \rceil}} \geq \left(\frac{\zeta n-h+1}{n-h+1}\right)^h \geq (\zeta/2)^h\quad \forall \omega \in \Omega,$$
provided $n \geq 2h/\zeta$.

Let $X$ be the non-negative random variable defined by $X(B) = \alpha_J^B$. Then we have $\mathbb{E}[X] = \mu(\F)$, $f(\mathbb{E}[X]) = \mu(\F) \log \mu(\F) = \phi(\F,\varnothing)$ and $\mathbb{E}[f(X)] = \phi(\F,J)$. The fact that $\F$ is $(\eta,h)$-potentially stable implies that $\mathbb{E}[f(X)] < f(\mathbb{E}[X]) + \eta$. Corollary \ref{cor:foxconn} (with $\lambda = (\zeta/2)^h$ and $C=1$) implies that $\|X - \mathbb{E}[X]\|_{\infty} < \delta$, provided $\eta$ is sufficiently small depending on $\delta,h$ and $\zeta$. This in turn implies that $|\mu(\F_J^B) - \mu(\F)| < \delta$ for any $B \subset J$. It follows that $\F$ is $(\delta,h)$-slice-quasirandom, as required.
\end{proof}

Armed with this lemma, we can now prove Theorem \ref{thm:weak-reg}.

\begin{proof}[Proof of Theorem \ref{thm:weak-reg}]
Let $\eta = \eta(\delta,h,\zeta/2)$ be as in the statement of Lemma \ref{lem:potential-qr}, and let $n_0 \in \mathbb{N}$ to be chosen later, depending on $\delta,h,\zeta$ and $\epsilon$. Given a set $S\subset \left[n\right]$, we define a partition of $\p(S)$ into three parts:
\begin{itemize}
\item We let $\g_{S}\subset \p\left(S\right)$ be the family of all sets $B\subset S$, such that $\mu\left(\f_{S}^{B}\right)>\frac{\epsilon}{2}$, and such that the family $\f_{S}^{B}$ is $\left(\eta,h\right)$-potentially
stable. We call these the `good' sets.
\item We let $\b_{S} \subset \p(S)$ be the family of all sets $B\subset S$, such that $\mu\left(\f_{S}^{B}\right)>\frac{\epsilon}{2}$, and such
that the family $\f_{S}^{B}$ is {\em not} $\left(\eta,h\right)$-potentially
stable. We call these the `bad' sets.
\item We let $\mathcal{E}_{S} \subset \p(S)$ be the family of all sets $B\subset S$, such
that $\mu\left(\f_{S}^{B}\right)\leq\frac{\epsilon}{2}$. We call these the `exceptional' sets.
\end{itemize}

By Lemma \ref{lem:potential-qr} (applied with $\zeta/2$ in place of $\zeta$), for each $B \in \G_S$, the family $\f_{S}^{B}$ is $\left(\delta,h\right)$-slice-quasirandom, provided $\eta$ is sufficiently small depending on $\zeta,\delta$ and $h$, and provided $n$ is sufficiently large depending on $\zeta$, $h$ and $|S|$. It suffices to show that there exists a set $S$ of size bounded from above in terms of $\zeta,\delta,h$ and $\epsilon$, such that $\mu\left(\f\backslash\left\langle \g_{S}\right\rangle \right)<\epsilon$. The next claim says that if this does not hold, then $S$ can be replaced by
a set $S' \supset S$ of size at most $h\cdot2^{\left|S\right|}$, such that
$\phi\left(\f,S'\right)$ is significantly larger than $\phi\left(\f,S\right)$.
\begin{claim}
\label{claim:split}
For each $S\subset \left[n\right]$, we either have $\mu\left(\f\backslash\left\langle \g_{S}\right\rangle \right)<\epsilon$,
or else there exists a set $S' \supset S$ with $|S'| \leq h2^{\left|S\right|}+\left|S\right|$,
such that
\[
\phi\left(\f,S'\right)\ge\phi\left(\f,S\right)+\eta\epsilon/2.
\]
\end{claim}
\begin{proof}
Suppose that $\mu\left(\f\backslash\left\langle \g_{S}\right\rangle \right) \geq \epsilon$.
We have
\begin{align*}
\epsilon & \leq \mu\left(\f\backslash\left\langle \g_{S}\right\rangle \right)\\
& =\sum_{B \subset S:\atop B\notin\g_{S}}\mu_{\left(n,k,S\right)}\left(B\right)\mu\left(\f_{S}^{B}\right)\\
 & =\sum_{B\in\b_{S}}\mu_{\left(n,k,S\right)}\left(B\right)\mu\left(\f_{S}^{B}\right)+\sum_{B\in\mathcal{E}_{S}}\mu_{\left(n,k,S\right)}\left(B\right)\mu\left(\f_{S}^{B}\right)\\
 & \le\mu_{\left(n,k,S\right)}\left(\b_{S}\right)+\epsilon/2,
\end{align*}
and therefore
\[
\mu_{\left(n,k,S\right)}\left(\b_{S}\right)\ge \epsilon/2.
\]
 For each $B\in\b_S$, let $S_{B} \subset [n] \setminus S$ such that $|S_B|\leq h$ and
\[
\phi\left(\f_{S}^{B},S_{B}\right)\ge\phi\left(\f_{S}^{B},\varnothing\right)+\eta.
\]
Writing
\[
S'=\left(\bigcup_{B\in\b_{S}}S_{B}\right)\dot\cup\ S,
\]
we have
\begin{align*}
\phi\left(\f,S'\right)-\phi(\f,S) & =\underset{B\sim\mu_{(n,k,S)}}{\mathbb{E}}\left[\phi\left(\f_{S}^{B},S'\backslash S\right)-\phi\left(\f_S^B,\varnothing\right)\right]\\
 & \ge\underset{B\sim\mu_{(n,k,S)}}{\mathbb{E}}\left[\mathbf{1}_{B\in\b_S}\left(\phi\left(\f_{S}^{B},S_{B}\right) -\phi\left(\f_S^B,\varnothing\right)\right)\right]\\
 & \ge\underset{B\sim\mu_{(n,k,S)}}{\mathbb{E}}\left[\mathbf{1}_{B\in\b_S} \cdot \eta\right]\\
 & = \eta \cdot \mu_{(n,k,S)}(\b_S)\\
 & \ge\eta\epsilon/2.
\end{align*}
 This proves the claim.
\end{proof}
Let $J_{0}=\varnothing$. By Claim \ref{claim:split}, we either have $\mu\left(\f\backslash\left\langle\g_{J_{0}}\right\rangle\right)<\epsilon$,
or else there exists $J_{1} \subset [n]$ such that $|J_1| \leq h\cdot 2^{0}+0$ and
\[
\phi\left(\f,J_{1}\right)\ge\phi\left(\f,J_{0}\right)+\eta\epsilon/2.
\]
We now repeat the process with the set $J_{1}$, and so on, producing, for each $m \in \mathbb{N}$ with $m \leq 2/(e \eta \epsilon)$, a set $J_{m} \subset [n]$, such that either
\[
\mu\left(\f\backslash\left\langle \g_{J_{m-1}}\right\rangle \right)<\epsilon,
\]
or else
\[
\phi\left(\f,J_{m}\right)\ge\phi\left(\f,\g_{J_{m-1}}\right)+\eta\epsilon/2,
\]
Since
\[
\phi\left(\f,\varnothing\right)\ge-\frac{1}{e},
\]
there exists $m\le\frac{2}{e\eta\epsilon}$ such that
\[
\mu\left(\f\backslash\left\langle \g_{J_{m-1}}\right\rangle \right)<\epsilon.
\]
The size of $J_{m-1}$ is bounded above by a constant depending only upon $\eta$, $h$ and
$\epsilon$. By definition, $\F_{J_{m-1}}^B$ is $(\eta,h)$-potentially-stable, for each $B \in \G_{J_{m-1}}$. Provided $n_0$ is sufficiently large (depending on $\zeta$, $\eta$, $h$ and $\epsilon$), we may apply Lemma \ref{lem:potential-qr} with $\zeta/2$ in place of $\zeta$, $n-|J_{m-1}|$ in place of $n$ and $k-|B|$ in place of $k$, implying that $\F_{J_{m-1}}^B$ is $(\delta,h)$-slice-quasirandom for each $B \in \G_{J_{m-1}}$. Hence, we may take $J = J_{m-1}$, completing the proof.
\end{proof}

\begin{remark}
\label{remark:counting}
We note that there is no general `counting lemma' for $(\eta,h)$-slice-quasirandom families, for hypergraphs with a bounded number of edges --- unsurprisingly perhaps, given the relative weakness of the constraint. To see this, let $n$ be even, let $k=n/2$, and consider a random family $\F$ produced by including exactly one of $S$ and $[n] \setminus S$ (with probability $1/2$ each), independently at random for each pair $\{S,[n] \setminus S\} \subset \binom{[n]}{n/2}$. Clearly, $\mu(\F) = \tfrac{1}{2}$, and $\F$ contains no copy of the hypergraph consisting of two disjoint edges. Moreover, for any $J \subset [n]$ and any $B \subset J$, $|\F_J^B| \sim \Bin(\binom{n-|J|}{k-|B|},1/2)$, and therefore by a Chernoff bound,
$$\Pr[|\mu(\F_J^B) - 1/2| \geq \eta] < 2\exp\left(-\frac{2}{3}\eta^2 \binom{n-|J|}{k-|B|}\right).$$
Hence, using a union bound, the probability that there exists $J \subset [n]$ with $|J| \leq h$ and $B \subset J$ such that $|\mu(\F_J^B) - 1/2| \geq \eta$ is at most
$$\sum_{J \subset [n]:\atop |J| \leq h} \sum_{B \subset J} 2\exp\left(-\frac{2}{3}\eta^2 \binom{n-|J|}{k-|B|}\right) \leq 2 \cdot 2^h \cdot \left(\sum_{i=1}^{h} {n \choose i}\right) \exp\left(-\frac{2}{3} \binom{n-h}{n/2-h} \right) = o(1),$$
as $n \to \infty$ for any fixed $\eta$ and $h$. Therefore, $\F$ is $(\eta,h)$-slice-quasirandom with high probability.
\end{remark}

To obtain results on families with a forbidden intersection-size, we will need the following property of pairs of slice-quasirandom families.
\begin{lemma}
\label{lem:qrc} For any
$\epsilon,\zeta\in\left(0,1\right)$ and any $t \in \mathbb{N}$,
there exist $\eta = \eta(\epsilon)>0$, $h_0 = h_0(\epsilon,\zeta,t) \in \mathbb{N}$ and $n_0 = n_0(\epsilon,\zeta,t) \in \mathbb{N}$ such that the following holds. Let $t-1 \leq k_{1},k_{2}\le\left(\frac{1}{2}-\zeta\right)n$,
and let $\A \subset \binom{\left[n\right]}{k_{1}},\ \B\subset \binom{\left[n\right]}{k_{2}}$
be $\left(\epsilon/5,h_0\right)$-slice-quasirandom families, with $\mu(\A) \geq \epsilon$ and $\mu(\B)\geq \epsilon$. Then there exist $A\in\A$ and $B\in\B$
such that $\left|A\cap B\right|=t-1$. (We may take $\eta = \epsilon/5$.)
\end{lemma}

This lemma says that if $\A$ and $\B$ are sufficiently slice-quasirandom, and have uniformity bounded away from $n/2$, then we can find a pair of sets $A \in \A,\ B \in \B$ with any bounded intersection-size. The idea of the proof is to use the slice-quasirandomness property to reduce to the case of $t=1$, where we can apply a result about cross-intersecting families (Lemma \ref{lem:dfv}). Note that Lemma \ref{lem:dfv} immediately implies the $t=1$ case of Lemma \ref{lem:qrc}.
\begin{proof}[Proof of Lemma \ref{lem:qrc}]
Let $n \geq n_0$ and let $\A \subset \binom{[n]}{k_1}$ and $\B\subset \binom{\left[n\right]}{k_2}$ be $\left(\eta,h_0\right)$-slice-quasirandom families, such that $\mu\left(\A\right) \geq \epsilon$ and $\mu\left(\B\right)\ge\epsilon$,
where $\eta = \epsilon/5$, and $h_0,n_0$ are to be chosen later.

It is easy to check that the families $\A':=\A_{\left[t-1\right]}^{\left[t-1\right]},\ \B':=\B_{\left[t-1\right]}^{\left[t-1\right]}$
are $\left(2\eta,h_0-t+1\right)$-slice-quasirandom families with
\[
\mu\left(\A'\right),\mu\left(\B'\right)\ge\epsilon-\eta.
\]
 By Lemma \ref{lem:dfv}, if $\A'$ and $\B'$ were cross-intersecting, then there would exist $S \subset [n]$ with $|S| \leq s(\zeta,\epsilon-\eta)$ and $\mu((\A')_{S}^{\varnothing}) < (\epsilon-\eta)/2$, provided $n_0$ is sufficiently large depending on $\epsilon$ and $\zeta$. But then
 $$\mu(\A') - \mu((\A')_{S}^{\varnothing}) > \epsilon-\eta - (\epsilon-\eta)/2 \geq 2\eta,$$
contradicting the fact that $\A'$ is $(2\eta, h_0-t+1)$-slice-quasirandom, provided $h_0-t+1 \geq s(\epsilon-\eta,\zeta)$. Hence, there exist $C\in\A'$ and $D\in\B'$ such that $C \cap D = \varnothing$. We have $C\cup\left[t-1\right]\in\A$, $D\cup\left[t-1\right]\in\B$ and $|(C \cup [t-1]) \cap (D \cup [t-1])|=t-1$, as required.
\end{proof}

\section{Proof of our forbidden intersection theorem}
\label{sec:main-result-es}

\subsection{Approximations by juntas}

We can now prove that if $k/n$ is bounded away from $0$ and $1/2$, and $\F \subset \binom{[n]}{k}$ is a family such that no pair of sets in $\F$ have intersection of size $t-1$, then $\F$ is approximately contained within a $t$-intersecting junta.

\begin{theorem}
\label{thm:junta-approx}
For any $\epsilon,\zeta>0$ and $t\in \mathbb{N}$ there exists $j=j(t,\zeta,\epsilon) \in \mathbb{N}$ and $n_1= n_1(t,\zeta,\epsilon) \in \mathbb{N}$ such that the following holds. Let $n \geq n_1$, let $\zeta n \leq k \leq \left(\frac{1}{2}-\zeta\right)n$, and
let $\f\subset \binom{\left[n\right]}{k}$ such that no two sets in $\f$ have intersection of size $t-1$. Then there
exists a $t$-intersecting $j$-junta $\j$ such that $|\f\backslash\j|<\epsilon\binom{n}{k}$.
\end{theorem}
The proof uses our `weak regularity lemma' (Theorem \ref{thm:weak-reg}) to find a $J$-junta $\j$ in which the family $\f$ is approximately contained, and such that for each $B \subset J$ with $B \in \J$, the slice $\F_J^B$ is highly slice-quasirandom and not too small; we then use Lemma \ref{lem:qrc} to show that the junta $\j$ must be $t$-intersecting.
\begin{proof}[Proof of Theorem \ref{thm:junta-approx}]
Let $t \in \mathbb{N}$, let $\epsilon,\zeta >0$, let $\eta = \epsilon/10$ and let $h,n_1 \in \mathbb{N}$ to be chosen later (depending on $t,\zeta$ and $\epsilon$). Let $n \geq n_1$, let $\zeta n \leq k \leq (\tfrac{1}{2}-\zeta)n$, and let $\f\subset \binom{\left[n\right]}{k}$ be a family containing no pair of sets whose intersection has size $t-1$.

By Theorem \ref{thm:weak-reg}, there exists $j = j(\zeta,\eta,h,\epsilon) \in \mathbb{N}$, a set $J \subset [n]$ with $|J| \leq j$, and a $J$-junta $\j = \langle \G \rangle$ (i.e., $\G \subset \p(J)$), such that $\mu\left(\f\backslash\j\right)<\epsilon,$
and such that for each $B\in\G$, the family $\f_{J}^{B}$ is an $\left(\eta,h\right)$-slice-quasirandom
family with $\mu(\F_J^B) \geq \epsilon/2$. It suffices to show that the junta $\j$ is $t$-intersecting.

Suppose for a contradiction that there exist $A_{1},A_{2}\in\J$
such that $|A_{1}\cap A_{2}|<t$. Then there exist $B_{1},B_{2} \in \G$
and $C_{1},C_{2} \subset \left[n\right]\backslash J$, such that
\[
A_{1}=B_{1}\cup C_{1},A_{2}=B_{2}\cup C_{2}.
\]
Note that the families $\F_{J}^{B_1}$ and $\F_{J}^{B_2}$ are each $(\eta,h)$-slice quasirandom with measure at least $\epsilon/2$. Write $\left|B_{1}\cap B_{2}\right|=:t' \leq t-1$. Provided $h \geq \max_{t'' \in [t-1] \cup \{0\}} h_0(\epsilon/2,\zeta/2,t'')$ and $n_1$ is sufficiently large depending on $t$, $\zeta$ and $\epsilon$, Lemma \ref{lem:qrc} (applied with $\A = \F_{J}^{B_1}$, with $\B = \F_{J}^{B_2}$, with $\epsilon/2$ in place of $\epsilon$ and with $\zeta/2$ in place of $\zeta$) implies that there exist $D_{1}\in\f_{J}^{B_{1}}$ and $D_{2}\in\f_{J}^{B_{2}}$ such that
\[
\left|D_{1}\cap D_{2}\right|=t-1-t'.
\]
This is a contradiction, since $B_{1}\cup D_{1} \in \f$, $B_{2}\cup D_{2} \in \f$
and $|(B_{1}\cup D_{1}) \cap (B_{2}\cup D_{2})| = t-1$.
\end{proof}

\begin{comment}
Let $\f\subseteq\binom{\left[n\right]}{k}$ be some family that doesn't
contain two sets whose intersection is of size exactly $t-1$. In
this section we apply Theorem \ref{thm:weak-reg} to show
that $\f$ can be approximated be a $t$-intersecting junta. As mentioned
the idea of the proof is to take the junta $\j=\left\langle \{ B_{1},\ldots,B_{l} \} \right\rangle $,
and the set $J$ of Theorem \ref{thm:weak-reg}, and to show
that this junta is $t$-intersecting. Equivantly, we need to show
that the family $\g:=\left\{ B_{1},\ldots,B_{l}\right\} \subseteq\p\left(J\right)$
is $t$-intersecting. If it wasn't $t$-intersecting, then there would
be sets $E_{1},E_{2}\in\g$, such that $\left|E_{1}\cap E_{2}\right|<t$.
Now this would imply that the slice-quasirandom families $\f_{J}^{E_{1}},\f_{J}^{E_{2}}$
satisfy the restriction of having no two sets $D_{1}\in\f_{J}^{E_{1}},D_{2}\in\f_{J}^{E_{2}}$,
such that $\left|D_{1}\cap D_{2}\right|=t-1-\left|E_{1}\cap E_{2}\right|$,
for otherwise the sets $D_{1}\cup E_{1},D_{2}\cup E_{2}$ are two
sets in $\f$ whose intersection is of size $t-1$. The hard part
of the proof would then be to show that the slice-quasirandom families
$\f_{J}^{E_{1}},\f_{J}^{E_{2}}$ cannot satisfy such a restriction. The proof of the last part is based on Theorem \ref{thm:fj}, the $p$-biased version of Friedgut's junta theorem.
\end{comment}

\subsection{A stability result for the forbidden intersection problem}

We now apply Theorems \ref{thm:main-ak-stability} and \ref{thm:junta-approx} to obtain the following stability
version of Theorem \ref{thm:main-result-es}. For brevity, we say that a family $\f \subset \binom{[n]}{k}$ is a {\em Frankl family} if there exists a set $S \in \binom{[n]}{t+2r}$ such that $\f = \{A \in \binom{\left[n\right]}{k}:\ \left|A\cap S\right|\ge t+r\}$. If $t \in \mathbb{N}$ and $\zeta >0$ are fixed and $\zeta n \leq k \leq (1/2-\zeta)n$, we say such a Frankl family is a {\em Frankl junta} if $r$ is bounded from above in terms of $t$ and $\zeta$.
\begin{theorem}
\label{thm:stability-es} For any $\epsilon,\zeta>0$ and any $t \in \mathbb{N}$, there
exists $\delta>0$ and $n_0 \in \mathbb{N}$ such that the following holds. Let $n \geq n_0$, let $\zeta n \leq k \leq (\tfrac{1}{2}-\zeta)n$,
and let $\a\subset \binom{\left[n\right]}{k}$ be a family that
does not contain two sets whose intersection is of size $t-1$. If $\left|\a\right|\ge f\left(n,k,t\right)-\delta\binom{n}{k}$,
then there exists a Frankl family $\f \subset \binom{[n]}{k}$ such that $\mu\left(\a\backslash\f\right)<\epsilon$.
Moreover, $\f$ can be taken to be a $j$-junta, where $j = j(t,\zeta) \in \mathbb{N}$.
\end{theorem}

\begin{proof}
 Let $\delta = \delta(t,\zeta,\epsilon)>0$ and $n_0 = n_0(t,\zeta,\epsilon) \in \mathbb{N}$ to be chosen later. Let $n \geq n_0$, let $\zeta n \leq k \leq (1/2-\zeta)n$, let $\a \subset \binom{\left[n\right]}{k}$ be a family that does not contain a pair of sets with intersection of size $t-1$, and suppose that $\left|\a\right|\geq f\left(n,k,t\right)-\delta\binom{n}{k}$. By Theorem \ref{thm:junta-approx}, provided $n_0$ is sufficiently large depending on $t,\zeta,\epsilon$ and $\delta$, there exists
a $t$-intersecting family $\j \subset \pn$ such that $\mu\left(\a\backslash\j\right)<\min\left\{ \delta,\frac{\epsilon}{2}\right\}$. In particular, we have
\[
\left|\j\right|\ge\left|\a\right|-\delta\binom{n}{k}\ge f\left(n,k,t\right)-2\delta\binom{n}{k}.
\]
Provided $\delta$ is sufficiently small depending on $t,\zeta$ and $\epsilon$, Theorem \ref{thm:main-ak-stability} implies that there exists a Frankl family $\f$, which is an $O_{\zeta,t}(1)$-junta, such that $\mu\left(\j\backslash\f\right)<\frac{\epsilon}{2}$. We have
\[
\mu\left(\a\backslash\f\right)\le\mu\left(\a\backslash\j\right)+\mu\left(\j\backslash\f\right)<\epsilon,
\]
completing the proof.
\end{proof}

\subsection{The Frankl families are locally extremal}

In the previous subsection, we showed that if $\a\subset \binom{\left[n\right]}{k}$
is a family that does not contain two sets whose intersection is
of size $t-1$, and if $\left|\a\right|$ is close to $f\left(n,k,t\right)$,
then $\a$ has small symmetric difference with a Frankl junta. In this subsection, we show that such a family $\A$ is no larger than a Frankl junta. This will complete the proof of Theorem \ref{thm:main-result-es}. The following lemma will be a key tool.

\begin{lemma}
\label{lem:max-juntas}
For any $\zeta >0$ and any $j,t \in \mathbb{N}$, there
exist $\epsilon_{0} = \epsilon_0(t,\zeta,j) >0$ and $n_{0} = n_0(t,\zeta,j) \in \mathbb{N}$ such that the following holds. Let $n\ge n_{0}$, and let $\zeta n \leq k \leq \left(\frac{1}{2}-\zeta\right)n$.
Let $\f\subset \binom{\left[n\right]}{k}$ be a family not containing a pair of sets whose intersection is of size $t-1$. Let $J \in \binom{[n]}{j}$, and let $\g\subset \p\left(J\right)$ be a maximal
$t$-intersecting family. Suppose that $\mu\left(\f_{J}^{B}\right)>1-\epsilon_{0}$
for any $B\in\g$. Then $\mu\left(\f\right)\le\mu\left(\left\langle \g\right\rangle \right)$,
with equality only if $\f=\left\langle \g\right\rangle $.\end{lemma}
\begin{proof}
Let $\delta=\max_{B\notin\g}\left\{ \mu\left(\f_{J}^{B}\right)\right\}$,
and let $\epsilon=\max_{A\in\g}\left(1-\mu\left(\f_{J}^{A}\right)\right)$. We observe the following.
\begin{claim}
\label{Last claim}There exists $c=c\left(t,\zeta,j\right)>1$
such that $\delta=O_{t,\zeta,j}\left(\epsilon^{c}\right)$.\end{claim}
\begin{proof}
Let $B \notin \g$ such that $\mu\left(\f_{J}^{B}\right)=\delta$.
Since $\g \subset \p\left(J\right)$ is maximal $t$-intersecting, there exists
$A\in\g$ such that $\left|A\cap B\right|<t$. By averaging, there exists $C\subset \left[n\right]\backslash J$
with $|C| = t-1-\left|A\cap B\right|$ and $\mu\left(\f_{J\cup C}^{B\cup C}\right)\ge\mu\left(\f_{J}^{B}\right) = \delta$. Note that $\f_{J \cup C}^{A \cup C}$ and $\f_{J \cup C}^{B \cup C}$ are cross-intersecting, otherwise $\f$ would contain two sets whose intersection is of size $t-1$. Note also that
\[
1-\mu\left(\f_{J\cup C}^{A\cup C}\right) \leq \frac{\binom{n-j}{k-|A|}}{\binom{n-j-|C|}{k-|A|-|C|}}\left(1-\mu\left(\f_{J}^{A}\right)\right) = O_{t,\zeta,j}\left(1-\mu\left(\f_{J}^{A}\right)\right)=O_{t,\zeta,j}\left(\epsilon\right).
\]
The claim now follows by applying Lemma \ref{lem:cross-intersecting} to $\f_{J \cup C}^{A \cup C}$ and $\f_{J \cup C}^{B \cup C}$.
\end{proof}
Since $\mu(\F) = \sum_{B \subset J} \mu_{(n,k,J)}(B) \mu(\F_J^B)$, and since $\mu_{(n,k,J)}(B) = \Omega_{\zeta,j}(1)$ for all $B \subset J$ (provided $n$ is sufficiently large depending on $\zeta$ and $j$), we have
\[
\mu\left(\f\right) \leq \mu\left(\left\langle \g\right\rangle \right)+\delta-\Omega_{t,\zeta,j}\left(\epsilon\right).
\]
Therefore, by Claim \ref{Last claim}, we have
\[
\mu\left(\f\right)\le\mu\left(\left\langle \g\right\rangle \right)+O_{t,\zeta,j}\left(\epsilon^{c}\right)-\Omega_{t,\zeta,j}\left(\epsilon\right)
\]
 for some $c>1$. Provided $\epsilon_{0}$ is sufficiently small (depending on $t,\zeta$ and $j$),
this implies that either $\mu\left(\f\right) < \mu\left(\left\langle \g\right\rangle \right)$
or $\f=\left\langle \g\right\rangle$, proving the lemma.
\end{proof}

We may now prove Theorem \ref{thm:main-result-es}.
\begin{proof}[Proof of Theorem \ref{thm:main-result-es}.]
 Given $\zeta >0$ and $t \in \mathbb{N}$, we choose $j=j\left(t,\zeta\right) \in \mathbb{N}$ as in Theorem
\ref{thm:stability-es}, we choose $\epsilon_{0}=\epsilon_{0}\left(t,\zeta,j\right)>0$ as in Lemma \ref{lem:max-juntas}, and we let $\epsilon_{1}=\epsilon_{1}\left(\epsilon_{0},\zeta,j\right)>0$ and $n_0 = n_0(t,\zeta) \in \mathbb{N}$ to be chosen later.

Let $n \geq n_0$, let $\zeta n \leq k \leq (1/2-\zeta)n$, let $\a\subset \binom{\left[n\right]}{k}$ be a family that does not
contain two sets whose intersection is of size $t-1$, and suppose
that $\left|\a\right|\ge f\left(n,k,t\right)$. We will show that $\a$ is a Frankl family. By Theorem \ref{thm:stability-es}, provided $n_0$ is sufficiently large depending on $t$, $\zeta$ and $\epsilon_1$,
there exists a Frankl family $\f$ such that $\f$ is a $j$-junta, and
\[
\mu\left(\a\backslash\f\right)<\epsilon_{1}.
\]

Let $J \in \binom{[n]}{j}$ and let $\g \subset \p(J)$ such that $\f = \langle \g \rangle$. For any $B\in\g$, we have
\[
\mu_{(n,k,J)}\left(B\right)\left(1-\mu\left(\a_{J}^{B}\right)\right)\le \mu(\f \backslash \a) \le \mu\left(\a\backslash\f\right)<\epsilon_{1}.
\]
Provided $n$ is sufficiently large depending on $\zeta$ and $j$, we have $\mu_{(n,k,J)}(B) = \Omega_{\zeta,j}(1)$ for all $B \subset J$. Hence, provided $\epsilon_{1}$ is sufficiently small depending on $\epsilon_0$, $\zeta$ and $j$, we have $\mu\left(\a_{J}^{B}\right) > 1-\epsilon_{0}$ for all $B \in \G$. Therefore, by Lemma \ref{lem:max-juntas}, provided $n_0$ is sufficiently large depending on $t$, $\zeta$ and $j$, we have $\mu(\a) \leq \mu(\f)$, with equality only if $\a=\f$, proving the theorem.
\end{proof}

\section{Conclusion and open problems}
\label{sec:conc}
For fixed $t \in \mathbb{N}$, the results in this paper, combined with the previous results mentioned in the Introduction, resolve the Erd\H{o}s-S\'os problem (i.e., Problem \ref{prob:es}) for $2t \leq k \leq (1/2 - o(1))n$. However, the problem remains unsolved for $k/n$ very close to $1/2$. We believe that new techniques will be required to tackle the case where $k/n \approx 1/2$.

It would also be interesting to determine the optimal dependence of $j = j(\zeta,\delta,h,\epsilon)$ on $\zeta,\delta,h$ and $\epsilon$, in Theorem \ref{thm:weak-reg}. As stated above, our proof gives $j \leq 2 \uparrow \uparrow 1/(\zeta^{O(h)} \delta^2 \epsilon)$.

\subsubsection*{Acknowledgements}

We are grateful to Gil Kalai for several helpful discussions, and to Yuval Filmus for suggesting a more elegant way of presenting the proof of Theorem~\ref{thm:biased-ak-stability}, one which we have adopted.

\end{document}